\documentclass[a4paper]{amsart}

\usepackage{mathrsfs}
\usepackage{amsmath,amssymb,latexsym,amsfonts,amscd}
\usepackage{hyperref}

\usepackage{cite}
\usepackage{amsthm}   
\usepackage{empheq}   
\usepackage{graphicx}
\usepackage{color}
\usepackage{titletoc}
\usepackage[titletoc]{appendix}
\usepackage{enumitem}
\usepackage{tikz}
\usetikzlibrary{shapes.geometric, arrows}
\usetikzlibrary{decorations,decorations.markings}
\usetikzlibrary{arrows,shapes,chains}

\numberwithin{equation}{section}

\newcommand{\bR}{{\mathbb R}}
\newcommand{\bN}{{\mathbb N}}

\newcommand{\bfP}{{\mathbf{P}}}
\newcommand{\bfQ}{{\mathbf{Q}}}

\newcommand{\cA}{{\mathcal{A}}}
\newcommand{\cD}{{\mathcal{D}}}
\newcommand{\cL}{{\mathcal{L}}}
\newcommand{\cH}{{\mathcal{H}}}

\newcommand{\sA}{{\mathscr{A}}}
\newcommand{\sE}{{\mathscr{E}}}
\newcommand{\sF}{{\mathscr{F}}}
\newcommand{\sG}{{\mathscr{G}}}
\newcommand{\sI}{{\mathscr{I}}}

\newcommand{\fm}{{\mathfrak{m}}}
\newcommand{\ft}{{\mathfrak{t}}}

\newtheorem{theorem}{Theorem}[section]
\newtheorem{lemma}[theorem]{Lemma}
\newtheorem{proposition}[theorem]{Proposition}
\newtheorem{corollary}[theorem]{Corollary}

\theoremstyle{definition}
\newtheorem{definition}[theorem]{Definition}

\theoremstyle{remark}
\newtheorem{remark}[theorem]{Remark}

\numberwithin{equation}{section}

\begin{document}

\title[Dirichlet forms and polymer models]{Dirichlet forms and polymer models based on stable processes}

\author{Liping Li}
\address{RCSDS, HCMS, Academy of Mathematics and Systems Science, Chinese Academy of Sciences, Beijing 100190, China.}
\email{liliping@amss.ac.cn}

\author{Xiaodan Li}
\address{Fudan University, Shanghai 200433, China.}
\email{}

\thanks{The first named author is partially supported by NSFC (No. 11688101 and 11801546) and Key Laboratory of Random Complex Structures and Data Science, Academy of Mathematics and Systems Science, Chinese Academy of  Sciences (No. 2008DP173182).}



\begin{abstract}
In this paper, we are concerned with polymer models based on $\alpha$-stable processes, where $\alpha\in (\frac{d}{2},d\wedge 2)$ and $d$ stands for dimension. They are attached with a delta potential at the origin and the associated Gibbs measures are parametrized by a constant $\gamma$ playing the role of inverse temperature. Phase transition exhibits with critical value $\gamma_{cr}=0$. Our first object is to formulate the associated Dirichlet form of the canonical Markov process $X^{(\gamma)}$ induced by the Gibbs measure for a globular state $\gamma>0$ or the critical state $\gamma=0$. Approach of Dirichlet forms also leads to deeper descriptions of probabilistic counterparts of globular and critical states. Furthermore, we will characterize the behaviour of polymer near the critical point from probabilistic viewpoint by showing that $X^{(\gamma)}$ is convergent to $X^{(0)}$ as $\gamma\downarrow 0$ in a certain meaning. 
\end{abstract}

\maketitle

\tableofcontents

\section{Introduction}\label{SEC1}

Polymers are chemical compounds consisting essentially of repeating units, called monomers. Real polymers are complex objects on their own, typically fluctuating in a solvent as well as with other portion of themselves.
The study of polymer models has been a very active area of research in mathematical physics for a long time. It began to develop in 1930's under the influence of chemical and biological applications. Research into these models led to a great number of impressive advances, from the explanation of rubber elasticity to the creation of the theory of helix-coil transitions in proteins and nucleic acids. Many  physically relevant problems on polymer chains have been outlined in e.g. a review of Lifschitz, Grosberg and Khokhlov \cite{bib16}. 

In simplified discrete models, the configuration of a polymer, i.e. the sequence of locations of the monomers, follows the trajectory of a random walk on a lattice; 
see e.g. \cite{bib15}. As a continuous generalization, a continuum polymer model was constructed by Cranston et al. (see e.g. \cite{bib8, bib9, bib10}) in the context of Brownian motions. Let us use a few lines to explain some details. It starts with a system of finite size $T$, which means the length of the polymer. Let $\Omega_T:=C([0,T], \mathbb{R}^d)$, i.e. the family of all continuous paths of size $T$ in $\bR^d$, be the configuration space of the system. Then the polymer model is described by a Gibbs ensemble at each inverse temperature $\beta$ ($\geq 0$), realized as a probability measure $\mathbf{P}_{\beta, T}$ on $\Omega_T$, which is also called a Gibbs measure. More precisely, the underlying probability measure $\bfP_{0,T}$ is identified with the Wiener measure on $\Omega_T$ in this model, and we also denote it by $\mathbf{P}_T$ in abbreviation. For $\beta>0$, $\mathbf{P}_{\beta,T}$ is determined by the so-called Hamiltonian $H_T$, which is given by a certain potential function $v$ on $\mathbb{R}^d$ in the following manner:
\begin{equation}\label{EQ1HOT}
	H_T(\omega)=-\int_0^T v(\omega(t))dt,\quad \omega\in \Omega_T. 
\end{equation}
In other words, 
\begin{equation}\label{EQ1PTO}
	\mathbf{P}_{\beta, T}(d\omega)=\frac{\exp\{-\beta H_T(\omega)\}}{Z_{\beta,T}}\mathbf{P}_T(d\omega) =\frac{\exp\{\beta \int_0^Tv(\omega(t))dt\}}{Z_{\beta,T}}\mathbf{P}_T(d\omega),
\end{equation}
where $Z_{\beta,T}:=\mathbf{E}_T\exp\{-\beta H_T\}$ is the so-called partition function. From probabilistic viewpoint, the phenomenon of phase transition is observed by letting $T\uparrow \infty$ with certain tactic. It is shown in \cite{bib8} that there is a critical value $\beta_{cr}$ such that for $\beta<\beta_{cr}$ and $\beta>\beta_{cr}$, the polymer manifests different behaviours and is called in the \emph{diffusive state} and in the \emph{globular state} respectively. In the former state, the canonical process induced by the limiting measure is nothing but Brownian motion. However in the latter state, the limiting measure induces another diffusion process enjoying a certain ergodic measure $\psi_\beta^2(x)dx$. From analytic viewpoint, the (self-adjoint) operator 
\[
	\mathcal{H}_\beta=\frac{1}{2}\Delta+\beta\cdot v: L^2(\mathbb{R}^d)\rightarrow L^2(\mathbb{R}^d),
\] 
where $\Delta$ is the Laplacian operator and $v$ is the potential function in \eqref{EQ1HOT}, plays an important role in characterizing the phase transition. In the case that $v\in C_c^\infty(\mathbb{R}^d)$ is non-negative and not identically equal to $0$, it is well known that the spectrum of $\mathcal{H}_\beta$ consists of the absolutely continuous part $(-\infty, 0]$ and at most a finite number of non-negative eigenvalues $\lambda_j(\beta)$, i.e. $\sigma(\mathcal{H}_\beta)=(-\infty, 0]\cup \{\lambda_j(\beta): 0\leq j\leq N\}$. We enumerate the eigenvalues in a decreasing order and particularly, $\lambda_0(\beta)=\max\{\lambda_j(\beta): 0\leq j\leq N\}$ if $\{\lambda_j(\beta)\}\neq \emptyset$. When $\beta\leq \beta_{cr}$, $\sup\sigma(\mathcal{H}_\beta)=0$. When $\beta>\beta_{cr}$, it holds that $\lambda_0(\beta)>0$ and $\beta\mapsto \lambda_0(\beta)$ is increasing and continuous with $\lim_{\beta \downarrow \beta_{cr}}\lambda_0(\beta)=0$ and $\lim_{\beta\uparrow \infty} \lambda_0(\beta)=\infty$ (see \cite[Lemma~4.1]{bib8}). These facts about $\sigma(\mathcal{H}_\beta)$ are another reflection of phase transition. It is worth noting that in a globular state, the density function $\psi_\beta$ appearing in the above ergodic measure is exactly the ground state of $\mathcal{H}_\beta$, i.e. its eigenfunction with eigenvalue $\lambda_0(\beta)$. In addition, $\lambda_0(\beta)$ coincides with the rate of growth of $Z_{\beta, T}$ (also called the free energy of the ensemble), i.e. $\lambda_0(\beta)=\lim_{T\uparrow \infty}\left(\log Z_{\beta,T}\right)/T$,    
and the asymptotics of $\lambda_0(\beta)$ as $\beta\downarrow \beta_{cr}$ demonstrate universality in that they depend only on dimension (see \cite[Theorem~6.1]{bib8}).

More interestingly, another particular and significant case with $v=\delta_0$, i.e. the delta function at the origin, is explored in e.g. \cite{bib6, bib9} and similar phase transition appears only for $d=3$. Note that $\mathcal{H}_\beta$ are not self-adjoint any more and should be replaced by self-adjoint extensions, parametrized by a constant $\gamma\in\{-\infty\}\cup \mathbb{R}$ as shown in \cite[Theorem~2.1]{bib9}, of $\frac{1}{2}\Delta$ restricted to $C_c^\infty(\mathbb{R}^3\setminus \{0\})$. Denote the family of all these self-adjoint extensions by $\{\mathcal{L}_\gamma: \gamma\in \mathbb{R} \text{ or }-\infty\}$. Meanwhile, the Hamiltonian should be understood as a limit $-\lim_{\varepsilon\downarrow 0}\int_0^T A_\varepsilon \cdot 1_{(-\varepsilon,\varepsilon)}(\omega_t)dt$ in a certain manner, where $A_\varepsilon$ ($\uparrow \infty$ as $\varepsilon \downarrow 0$) is a constant depending on $\gamma$. This parameter $\gamma$ plays the role of inverse temperature in associated Gibbs measure, which exhibits a phase transition with critical value $\gamma_{cr}=0$, and $\gamma=-\infty$ corresponds to the underlying case. Indeed, in the diffusive state $\gamma<0$, $\mathbf{P}_{\gamma, T}$ converges to the Wiener measure under suitable scaling as $T\uparrow \infty$ and $\sigma(\mathcal{L}_\gamma)=(-\infty, 0]$. Note that when $\gamma=-\infty$, $\mathcal{L}_\gamma$ is exactly the Laplacian operator. In the globular state $\gamma>0$, $\sigma(\mathcal{L}_\gamma)=(-\infty, 0]\cup \{\lambda_0(\gamma):=\gamma^2/2\}$ and the limiting process has the ergodic measure $\psi_\gamma^2(x)dx$, where $\psi_\gamma$ is the eigenfunction of $\mathcal{L}_\gamma$ with the solo eigenvalue $\lambda_0(\gamma)=\gamma^2/2$. The behaviours of polymer near $\gamma_{cr}=0$ are also analysed in \cite{bib9}. 

In this paper, we are concerned with a generalization based on $\alpha$-stable process $W^\alpha$ with $\alpha\in (0,2)$ of this continuum polymer model with delta potential. 
For the sake of brevity, we only consider the isotropic case, where the transition density of $W^\alpha$ is given by \eqref{EQ2PTX}. Particularly, $\Delta^{\alpha/2}:=-(-\Delta)^{\alpha/2}$, not $\frac{1}{2}\Delta^{\alpha/2}$, is the generator of $W^\alpha$. 
This generalization was first raised in \cite{bib5} mainly from analytic viewpoint. At a heuristic level, the analogical operator, denoted by $\mathcal{A}_\gamma$, of $\mathcal{L}_\gamma$ may be informally written as 
\begin{equation}\label{EQ1DAB}
	\cA_\gamma \doteq\Delta^{\alpha/2}+\beta_\gamma\cdot \delta_0,
\end{equation}
where  $\beta_\gamma$ is a certain constant depending on $\gamma$ (such that $\beta_{-\infty}=0$; see \cite[\S3]{bib5}) and $\gamma$ plays the role of inverse temperature. Strictly speaking, $\mathcal{A}_\gamma$ is a self-adjoint extension on $L^2(\mathbb{R}^d)$ of $\Delta^{\alpha/2}$ restricted to $C_c^\infty(\mathbb{R}^d\setminus \{0\})$. The rigorous statement is phrased in  \cite[Theorem~3.3]{bib5} for either of the following cases:  
\begin{itemize}
\item[(i)] $d=1$ and $\alpha>1$;
\item[(ii)] $d=1$ or $2$ and $\alpha=d$;
\item[(iii)] $d/2<\alpha<d$.
\end{itemize}  
In the cases (i) and (iii), phase transition exists with critical value $\gamma_{cr}=0$.  In the globular state $\gamma>0$, $\mathcal{A}_\gamma$ possesses a solo eigenvalue $\lambda_\gamma>0$. Meanwhile universality is demonstrated in the concrete expression of $\lambda_\gamma$ as presented in \cite[\S3.1 and \S3.3]{bib5}; but at this time, not only dimension $d$ but also $\alpha$ is involved. However in the case (ii), no phase transitions exhibit and every $\gamma$ corresponds to a globular state. 
The associated Gibbs measure (at $\gamma$) in a strict sense is also obtained in \cite{bib5}. In abuse of notation, we still denote the configuration space by $\Omega_T=D([0,T], \mathbb{R}^d)$, i.e. the family of all c\`adl\`ag paths in $\bR^d$. Fix a starting point $x$ of the underlying process $W^\alpha$ (we take $x=0$ in \eqref{EQ1HOT} tacitly). Let $p_\gamma(t,x,y)$ be the fundamental solution of 
\[
	\frac{\partial u}{\partial t}=\mathcal{A}_\gamma u.
\]
When $\gamma=-\infty$, we write $p$ for $p_{-\infty}$ (i.e. the transition density of $W^\alpha$) in abbreviation. 
Then the partition function, denoted by $Z_{\gamma, T}(x)$, and the Gibbs measure, denoted by $\mathbf{P}^x_{\gamma,T}$, are formulated as follows:
\begin{equation}\label{EQ1ZTX}
\begin{aligned}
&Z_{\gamma, T}(x)=\int_{\mathbb{R}^d}p_\gamma(T,x,y)dy,\\
&\bfP_{\gamma, T}^x(\{\omega\in \Omega_T: \omega(t_1)\in A_1,\cdots,\omega(t_n)\in A_n\}) \\
&\quad =Z_{\gamma,T}^{-1}(x)\int_{A_1 \times \cdots\times A_n \times \mathbb{R}^d} \prod_{1\leq i\leq n+1}p_\gamma(t_i-t_{i-1},x_{i-1},x_i) dx_{n+1}\cdots dx_1,
\end{aligned}
\end{equation}
where $A_1,\cdots, A_n$ are Borel subsets of $\mathbb{R}^d$, $x_0=x$ and $0=t_0<t_1<\cdots <t_n<t_{n+1}=T$. The case $\gamma=-\infty$ corresponds to the truncated $\alpha$-stable process and for $\gamma$ in the globular state, $\mathbf{P}^x_{\gamma, T}$ converges  to a probability measure $\mathbf{P}^x_\gamma$ on $\Omega:=D([0,\infty),\bR^d)$ inducing a canonical process with an ergodic measure $\psi_\gamma^2(x)dx$, where $\psi_\gamma$ is the ground state of $\mathcal{A}_\gamma$, as $T\uparrow \infty$ (see \cite[Theorem~4.1]{bib5}). To our knowledge, however, no analogical limits were obtained for a non-globular state. Note incidentally that the analogical approach of \eqref{EQ1ZTX} for $\mathcal{H}_\beta$ or $\mathcal{L}_\gamma$ is still available, see e.g. \cite{bib8}.

Our paper aims to study this polymer model of non-local type from probabilistic viewpoint by means of so-called Dirichlet forms. Although a conception in functional analysis, Dirichlet forms are closely linked with Markov processes in probability theory due to several seminal works by Fukushima in 1970's. It is now well known that a Dirichlet form satisfying so-called regular condition is always associated with a nice Markov process. The notions related to them are referred to \cite{bib19, bib11}. 

The first main result stated in Theorem~\ref{THM21} derives the associated Dirichlet form of the Markov process $X^{(\gamma)}$ induced by $\bfP^x_\gamma$ for a globular state $\gamma>0$, i.e. $X^{(\gamma)}_t(\omega):=\omega(t)$ for all $\omega\in \Omega$ and $t\geq 0$. This Dirichlet form denoted by $(\sE^{(\gamma)}, \sF^{(\gamma)})$ is regular with a core $C_c^\infty(\bR^d)$. Formulation of it builds a new bridge between analytic and probabilistic characterizations of globular states. On one hand, $\cA_\gamma$ is linked with the generator of $(\sE^{(\gamma)}, \sF^{(\gamma)})$ in \eqref{EQ2DAG}. On the other hand, more properties of $X^{(\gamma)}$ can be obtained by virtue of the theory of Dirichlet forms. For example, $X^{(\gamma)}$ is irreducible, recurrent and consequently ergodic as explained in \eqref{EQ2TTP}. 

As we see in \eqref{EQ2FGF}, the Dirichlet space $\sF^{(\gamma)}$ is a weighted Sobolev space of fractional order,  the weight function $\psi_\gamma$ in which is nothing but the resolvent density of $W^\alpha$ with parameter $\lambda_\gamma$ in \eqref{EQ2LGG}. Recall that the limiting Gibbs measure is not obtained for the critical case $\gamma=0$ in \cite{bib5}. However, the above Dirichlet form can be extended to the one with parameter $\gamma=0$ in a truly straightforward way: Replace the weight function $\psi_\gamma$ by $\psi_0:=u_0$, i.e. the Riesz potential kernel as presented in \eqref{EQ2UXP} (note that $\lim_{\gamma\downarrow 0}\lambda_\gamma=0$). This extension works for the third case $d/2<\alpha<d$, since the existence of $u_0$ relies on the transience of $W^\alpha$. Analogically we will show in Theorem~\ref{THM22} that this new Dirichlet form, denoted by $(\sE^{(0)},\sF^{(0)})$, is also regular with a core $C_c^\infty(\bR^d)$. Again its associated Markov process $X^{(0)}$ is irreducible and recurrent. But the symmetric measure of it is not finite, thus the ergodicity manifests a different limiting behaviour.  There are at least two evidences for that $X^{(0)}$ should correspond to the right Gibbs measure at $\gamma=0$. Firstly, the generator of $(\sE^{(0)},\sF^{(0)})$ is linked with $\cA_0$ in the same manner as \eqref{EQ2DAG}. Secondly, as will be explained later, $X^{(\gamma)}$ is convergent to $X^{(0)}$ as $\gamma\downarrow 0$ in a certain meaning. This continuity in $\gamma$ is in agreement with the behaviour of polymer near the critical point $\gamma_{cr}=0$ exhibited by the continuity of $\gamma\mapsto \lambda_\gamma$ near $\gamma_{cr}$.

\tikzstyle{decisions} =[text centered]

\tikzstyle{pre}=[<-,shorten <=1pt,>=stealth',semithick] 
\tikzstyle{rig}=[->,shorten <=1pt,>=stealth',semithick] 
\tikzstyle{mid}=[<->,shorten <=1pt,>=stealth',semithick] 

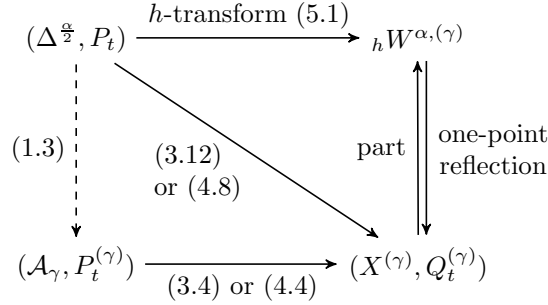
\begin{figure}
\centering
\begin{tikzpicture}[node distance=1cm]
\node[decisions](A){($\Delta^{\frac{\alpha}{2}}, P_t$)};
\node[decisions,below of=A,yshift=-2cm,xshift=0cm](B){$(\cA_\gamma, P^{(\gamma)}_t)$};
\node[decisions,right of=A,yshift=-0cm,xshift=3.5cm](C){${}_hW^{\alpha,(\gamma)}$};
\node[decisions,below of=C,yshift=-2cm,xshift=0cm](D){$(X^{(\gamma)}, Q^{(\gamma)}_t)$};
\node[decisions,below of=C,yshift=-2cm,xshift=0cm](E){};

\draw [pre] (C) -- node [anchor=south] {$h$-transform \eqref{EQ5HPG}} (A);
\draw [pre] (D) -- node [anchor=north] {\eqref{EQ3DAF} or \eqref{EQ4DAF}} (B);
\draw [>=latex, pre] ([xshift= 3pt] D.north) -- node [anchor=west,  text width=1.5cm, text centered] {one-point reflection} ([xshift= 3pt] C.south);
\draw [>=latex, rig] ([xshift=-3pt]  D) -- node [anchor=east, text width=0.7cm, text centered] {part} ([xshift=-3pt]  C);
\draw [rig] (A) -- node[xshift=-0.6cm, yshift=-0.3cm, text width=1.2cm]{\eqref{EQ3RGL} or \eqref{EQ4RLF}} (D);
\draw [rig, dashed] (A) -- node[xshift=-0.5cm, yshift=0cm]{\eqref{EQ1DAB}} (B);
\end{tikzpicture}
\caption{Globular or critical state under $h$-transform}\label{FIG1}
\end{figure}

Approach of Dirichlet forms has far more advantages in characterizing the probabilistic counterparts of globular and critical state. In \S\ref{SEC5}, we shall figure out a clear relation between $X^{(\gamma)}$ and $W^\alpha$ for every $\gamma\geq 0$. The first crucial fact is that the origin $0$ is of positive capacity relative to $X^{(\gamma)}$ as stated in Theorem~\ref{THM44}. This illustrates that $X^{(\gamma)}$ feels a strong attraction to the origin (see \eqref{EQ2PXG}), as usually appears in one-dimensional models. At a heuristic level, it is a reflection of that $\cA_\gamma$ has infinite potential at $0$ as we can see in \eqref{EQ1DAB}. In addition, the part process of $X^{(\gamma)}$ outside the origin, obtained by killing $X^{(\gamma)}$ once upon leaving $\bR^d\setminus \{0\}$, is identified with the $h$-transformed process, denoted by ${}_hW^{\alpha,(\gamma)}$, of $W^\alpha$ with $h=\psi_\gamma$. On the contrary, $X^{(\gamma)}$ is the unique one-point reflection of ${}_hW^{\alpha,(\gamma)}$ at $0$ in the sense of \cite{bib22}. As a result, we can summarize these probabilistic counterparts in a road map illustrated in Figure~\ref{FIG1}. Further interesting properties of $X^{(\gamma)}$ can be obtained from this alternative characterization. For example, $0$ is regular for itself with respect to $X^{(\gamma)}$; and the paths of $X^{(\gamma)}$ are not only c\`adl\`ag but also continuous at the moments $t$ when $X^{(\gamma)}_t=0$, although its associated Dirichlet form contains no diffusion part.

Another main result is that $X^{(\gamma)}$ converges to $X^{(0)}$ as $\gamma\downarrow 0$ in the following sense: Take $\gamma_n\downarrow 0$ and a non-negative function $\phi$ on $\bR^d$ such that
\[
	\phi/\psi_{\gamma_1}\in L^2(\bR^d),\quad \int_{\bR^d}\phi(x)dx=1,
\]
then $\bfP^\phi_{\gamma_n}(\cdot):=\int_{\bR^d}\bfP^x_{\gamma_n}(\cdot)\phi(x)dx$ converges to $\bfP^\phi_{0}(\cdot):=\int_{\bR^d}\bfP^x_0(\cdot)\phi(x)dx$ weakly on $\Omega$ endowed with the Skorohod topology as $n\rightarrow \infty$. The theory of Dirichlet forms plays an important role in the proof of it as well. Indeed, Mosco convergence of $(\sE^{(\gamma_n)},\sF^{(\gamma_n)})$ demonstrated in Theorem~\ref{THM63} leads to the convergence of finite dimensional distributions of $X^{(\gamma_n)}$, and to prove their tightness, an inequality concerning capacity, analysis of quasi-continuous functions and so-called Fukushima's decomposition are all employed.

Throughout this paper we will concentrate in globular and critical states for the case $\frac{d}{2}<\alpha<d$. For another two cases mentioned earlier, the characterization by means of Dirichlet forms is still available but only for globular states, since we cannot find a suitable substitution of $\psi_0$ when $W^\alpha$ is recurrent at present. The state $\gamma<0$ is not under consideration either, because the expected counterpart is nothing but $W^\alpha$. We wish to treat them in a future work. It is also worth pointing out that the Brownian case, i.e. $\alpha=2$ and $d=3$, has been explored by the first named author and his co-author in \cite{bib1}. Nevertheless, the current case of non-local type is much more involved, as we see the proofs of main results are far from routine.

The rest of this paper is organized as follows. In \S\ref{SEC2}, we will present the expression of associated Dirichlet form of $X^{(\gamma)}$ for every $\gamma\geq 0$. The proof will be postponed to \S\ref{SEC3} for globular states and to \S\ref{SEC4} for critical state. The road map illustrated in Figure~\ref{FIG1} will be completed with the help of a theorem in \S\ref{SEC5} providing an alternative characterization of $X^{(\gamma)}$ via $h$-transform. Finally, the weak convergence of $X^{(\gamma)}$ as $\gamma\downarrow 0$ will be proved in \S\ref{SEC6}. 

\subsection*{Notations}
Let us put some often used notations here for handy reference, though we may restate them when they appear.

The notation ``$:=$'' is read as ``to be defined as". For $x,\xi\in \bR^d$, $\langle x, \xi\rangle$ means the inner product between $x$ and $\xi$ and $|x|$ stands for the Euclidean norm of $x$. Given a domain $D\subset \bR^d$, the families $C_c(D), C_0(D)$ and $C_c^\infty(D)$ are those of all continuous functions on $D$ with compact support, all continuous functions on $D$ vanishing on the boundary of $D$ or at $\infty$ and all smooth functions on $D$ with compact support respectively. Given a continuous function $f$ with compact support, $\text{supp}[f]$ stands for its support, i.e. the closure of $\{x: f(x)\neq 0\}$. For every $r>0$, $B(r):=\{x: |x|<r\}$. The notation $\|\cdot\|_\infty$ means the supremum norm of a bounded function. Given a Hilbert space $H$, $\|\cdot\|_H$ stands for its norm and $(\cdot,\cdot)_H$ stands for its inner product. 
Given an operator $\mathcal{L}$, $\cD(\cL)$ stands for the domain of $\cL$ tacitly. 

The symbol $\lesssim$ (resp. $\gtrsim$) means that the left (resp. right) term is bounded by the right (resp. left) term multiplying a nonessential constant. In addition, $\dagger \approx \ddagger$ means that there is a nonessential constant $C>1$ such that $\frac{1}{C}\dagger \leq \ddagger \leq C\dagger$. There are several fixed constants throughout this paper: $c_{-\alpha,d}$, $c_{\alpha,d}$ and $c(\alpha,d)$ first appear in \eqref{EQ2UXP}, \eqref{EQ2DGH} and \eqref{EQ2LGG} respectively. Otherwise a constant attached with subscript means it depends on the terms in subscript. Note that almost all constants are relevant to $d$ and $\alpha$, and we ignore them in subscript if no confusions cause. 

The semigroup and resolvent of isotropic $\alpha$-stable process are denoted by $P_t$ and $U_\lambda$ respectively. Accordingly, the transition density and resolvent density are $p_t$ and $u_\lambda$. The Rieze potential kernel $u_0$ is given by \eqref{EQ2UXP}. For every finite $\gamma$, $P^{(\gamma)}_t$ stands for the semigroup associated with $\cA_\gamma$. The semigroup and resolvent of $X^{(\gamma)}$ for $\gamma\geq 0$ are denoted by $Q^{(\gamma)}_t$ and $R^{(\gamma)}_\lambda$ respectively. Meanwhile, $\psi_\gamma:=u_{\lambda_\gamma}$ where $\lambda_\gamma$ is given by \eqref{EQ2LGG}. 

The notions related to Dirichlet forms are referred to \cite{bib19, bib11}. Particularly, every function in a Dirichlet space is taken to be a quasi-continuous version if without other statements.


\section{Probabilistic counterparts of globular and cricital states}\label{SEC2}

Fix $\alpha \in (\frac{d}{2},2\wedge d)$ and let $W^\alpha=\{\Omega=D([0,\infty), \bR^d), (\bfP^x)_{x\in \bR^d}, (W^\alpha_t)_{t\geq 0}\}$ denote the isotropic $\alpha$-stable process on $\mathbb{R}^d$, i.e. $W^\alpha$ is a L\' evy process (see e.g. \cite{bib7}) whose transition density $p(t,x,y)=p(t,0, x-y)=:p_t(x-y)$ with respect to the Lebesgue measure is given by
\begin{equation}\label{EQ2PTX}
\hat{p}_t (\xi):=\int_{\mathbb{R}^d}e^{i\langle x,\xi\rangle}p_t(x)dx=e^{-t|\xi|^\alpha},\quad \xi\in \mathbb{R}^d.
\end{equation}
Its resolvent kernel $u_\lambda(x)$ for $\lambda>0$ is equal to
\begin{equation}\label{EQ2ULX}
	u_\lambda(x)=\int_0^\infty e^{-\lambda t}p(t,0,x)dt=\frac{1}{(2\pi)^d}\int_{\mathbb{R}^d}\frac{e^{-i\langle \xi,x\rangle}}{\lambda+|\xi|^\alpha}d\xi.
\end{equation}
Several properties of $u_\lambda$ are presented in Lemma~\ref{LMB}. Particularly,
 $u_\lambda\in L^2(\mathbb{R}^d)$ due to $\alpha>d/2$. In addition, $\alpha<d$ leads to
\begin{equation}\label{EQ2UXP}
	u_0(x):=\int_0^\infty p(t,0,x)dt=\uparrow \lim_{\lambda \downarrow 0} u_\lambda(x)=c_{-\alpha,d}\cdot |x|^{\alpha-d},
\end{equation}
where $c_{-\alpha,d}=\frac{2^{-\alpha}\Gamma(\frac{d-\alpha}{2})}{\pi^{d/2}\Gamma(\frac{\alpha}{2})}$ and $\Gamma$ is the so-called Gamma function. 
It is well known that the generator of $W^\alpha$ is $\Delta^{\alpha/2}$, which is symmetric with respect to the Lebesgue measure, and its associated Dirichlet form $(\mathscr{G},\mathcal{D}(\mathscr{G}))$ on $L^2(\mathbb{R}^d)$ is 
\begin{equation}\label{EQ2DGH}
\begin{aligned}
\mathcal{D}(\mathscr{G})&=H^{\alpha/2}(\bR^d)=\{f\in L^2(\bR^d): \sG(f,f)<\infty\}, \\
\sG(f,g)&=\frac{c_{\alpha,d}}{2}\int_{\bR^d\times \bR^d\setminus D}\frac{\left(f(x)-f(y)\right)\left(g(x)-g(y)\right)}{|x-y|^{d+\alpha}}dxdy, \quad f,g\in \cD(\sG),
\end{aligned}
\end{equation}
where $D$ is the diagonal of $\mathbb{R}^{d}\times\mathbb{R}^{d}$ and $c_{\alpha,d}=\frac{2^{\alpha}\Gamma(\frac{\alpha+d}{2})}{\pi^{\frac{d}{2}}|\Gamma(-\frac{\alpha}{2})|}$.

As mentioned in \S\ref{SEC1}, the self-adjoint extensions on $L^2(\bR^d)$ of $\Delta^{\alpha/2}$ restricted to $C_c^\infty(\bR^d\setminus \{0\})$ are parametrized by a constant $\gamma\in \{-\infty\}\cup \bR$. For $\gamma>\gamma_{cr}=0$ in a globular state, the corresponding self-adjoint extension $\cA_\gamma$ has a solo eigenvalue 
\begin{equation}\label{EQ2LGG}
	\lambda_\gamma=\left(\frac{\gamma}{c(\alpha,d)}\right)^{\alpha/(d-\alpha)},
\end{equation}
where $c(\alpha,d)=\frac{1}{(2\pi)^d}\int_{\bR^d}\frac{d\xi}{|\xi|^\alpha (1+|\xi|^\alpha)}$, with the eigenfunction
\[
	\psi_\gamma:=u_{\lambda_\gamma},
\]
where $u_{\lambda_\gamma}$ is given by \eqref{EQ2ULX} with $\lambda_\gamma$ in place of $\lambda$ (see \cite[\S3.3]{bib5}). One of the main purposes in this section is to present an alternative description of the probabilistic counterpart of this globular state by means of Dirichlet forms. To phrase the result, we prepare some notations. For $\gamma>0$, set $\fm_\gamma(dx):=\psi_\gamma(x)^2dx$, which is a finite measure, and define another operator
\begin{equation}\label{EQ2DAG}
\begin{aligned}
&\cD(\sA_\gamma):=\{f\in L^2(\bR^d, \fm_\gamma): f\cdot \psi_\gamma \in \cD(\cA_\gamma)\}, \\
&\sA_\gamma f := \frac{1}{\psi_\gamma} \cdot \cA_\gamma\left(f\cdot \psi_\gamma\right)-\lambda_\gamma f,\quad f\in \cD(\sA_\gamma). 
\end{aligned}
\end{equation}
It is not hard to find that $\mathscr{A}_\gamma$ is self-adjoint on $L^2(\bR^d, \fm_\lambda)$ with $\sA_\gamma 1=0$. 

\begin{theorem}\label{THM21}
Fix $\gamma>0$. Set $\fm_\gamma(dx)=\psi_\gamma(x)^2dx$ and let $X^{(\gamma)}=\{\Omega, \bfP^x_\gamma, X^{(\gamma)}_t\}$ be the process corresponding to the globular state at $\gamma$, i.e. $\Omega=D([0,\infty),\bR^d)$, $\mathbf{P}^x_\gamma$ is the probability measure on $\Omega$ mentioned below \eqref{EQ1ZTX} and $X^{(\gamma)}_t(\omega):=\omega(t)$ for $\omega\in \Omega$. Then $X^{(\gamma)}$ is $\fm_\gamma$-symmetric and associated with a regular Dirichlet form on $L^2(\bR^d, \fm_\gamma)$ as follows:
\begin{equation}\label{EQ2FGF}
\begin{aligned}
\sF^{(\gamma)} &=\{f\in L^2(\bR^d, \fm_\gamma): \sE^{(\gamma)}(f,f)<\infty\}, \\
\sE^{(\gamma)}(f,f) &=\frac{c_{\alpha,d}}{2}\int_{\bR^d\times \bR^d\setminus D}\frac{\left(f(x)-f(y)\right)^2}{|x-y|^{d+\alpha}}\psi_\gamma(x)\psi_\gamma(y)dxdy,\quad f\in \sF^{(\gamma)}. 
\end{aligned}\end{equation}
Furthermore, $C_c^\infty(\bR^d)$ is a core of $(\sE^{(\gamma)}, \sF^{(\gamma)})$, whose generator is $\sA_\gamma$ given by \eqref{EQ2DAG}. 
\end{theorem}

For the critical case $\gamma=\gamma_{cr}=0$, no probabilistic counterparts are obtained in \cite{bib5}. However, the analogues of \eqref{EQ2DAG} and \eqref{EQ2FGF} are still available. Indeed, set $\psi_0:=u_0$ in \eqref{EQ2UXP} and $\fm_0(dx):=\psi_0(x)^2dx$. Note that $\fm_0$ is positive Radon on $\bR^d$ since $\alpha>d/2$. Then the operator $\sA_0$ and the quadratic form $(\sE^{(0)}, \sF^{(0)})$ are well defined by letting $\gamma=\lambda_\gamma=0$ in \eqref{EQ2DAG} and \eqref{EQ2FGF} respectively. The analogical result of Theorem~\ref{THM21} below states the regularity of $(\sE^{(0)},\sF^{(0)})$, which leads to a probabilistic counterpart of the critical state $\gamma=0$, i.e. its associated Markov process denoted by $X^{(0)}:=\{\Omega, \bfP^x_0, X^{(0)}_t\}$.

\begin{theorem}\label{THM22}
The quadratic form $(\sE^{(0)}, \sF^{(0)})$ is a regular Dirichlet form on $L^2(\bR^d, \fm_0)$ with the generator $\sA_0$. Moreover, $C_c^\infty(\bR^d)$ is a core of $(\sE^{(0)}, \sF^{(0)})$. 
\end{theorem}

The proofs of these theorems are postponed to \S\ref{SEC3} and \S\ref{SEC4}. 
Instead, we point out two facts about $X^{(\gamma)}$ for $\gamma\geq 0$. The first one concerns their global properties. It will turn out in Propositions~\ref{THM45} and \ref{PRO41} that $X^{(\gamma)}$ is irreducible and recurrent. As a result, we can conclude that for $\gamma>0$ and any $x\in \bR^d$ (see \cite[Theorem~4.7.3]{bib11}),
\begin{equation}\label{EQ2TTP}
	\frac{1}{t}\int_0^t\bfP^x_\gamma (X^{(\gamma)}_s\in \cdot)ds \longrightarrow \pi_\gamma(\cdot):= \frac{\fm_\gamma(\cdot)}{\fm_\gamma(\bR^d)},\quad \text{weakly as }t\uparrow \infty. 
\end{equation} 
When $\gamma=0$, the probability measure on the left hand side is vaguely convergent to $0$ as $t\uparrow \infty$.  The second fact illustrates that $X^{(\gamma)}$ feels a strong attraction to the origin. Indeed, the capacity of $\{0\}$ relative to $\sE^{(\gamma)}$ is positive as shown in Theorem~\ref{THM44}. This property also leads to an alternative characterization of globular or critical state in \S\ref{SEC5} by means of Doob's well-known $h$-transform. Particularly it holds for $\sE^{(\gamma)}$-q.e. $x\in \bR^d$ (see \cite[Theorem~4.7.1]{bib11}), 
\begin{equation}\label{EQ2PXG}
	\bfP^x_{\gamma}(\sigma_0<\infty)=1,
\end{equation}
where $\sigma_0:=\inf\{t>0: X^{(\gamma)}_t=0\}$. 
Note incidentally that other singleton is always $\sE^{(\gamma)}$-polar.

\section{Globular states}\label{SEC3}

Fix $\gamma>0$. This section is mainly devoted to proving Theorem~\ref{THM21} and presenting some properties of $X^{(\gamma)}$.

\subsection{Proof of Theorem~\ref{THM21}}

This proof will be completed in several steps. 

\subsubsection{Step 1}\label{SEC311}
We prove $(\sE^{(\gamma)},\sF^{(\gamma)})$ is a Dirichlet form. To this end, set $$j_\gamma(x,dy):=|x-y|^{-(d+\alpha)}\psi_\gamma(y)\psi_\gamma(x)^{-1}dy.$$ By \cite[Example 1.2.4]{bib11}, it suffices to show
\begin{itemize}
\item[(j.1)] For any $\varepsilon >0$, $x\mapsto j_\gamma(x, \bR^d\setminus U_\varepsilon(x)):=\int_{y\in \bR^d\setminus U_\varepsilon(x)}j_\gamma(x,dy)$ is locally integrable with respect to $\fm_\gamma$, where $U_\varepsilon(x)$ is the $\varepsilon$-neighbourhood of $x$.
\item[(j.2)] $\int_{\bR^d} f(x)(j_\gamma g)(x)\fm_\gamma(dx)=\int_{\bR^d}(j_\gamma f)(x)g(x)\fm_\gamma(dx)$ for all $f, g\in\mathscr{B}^+(\bR^d)$, where $j_\gamma f(x):=\int f(y)j_\gamma(x,dy)$. 
\item[(j.3)] $\int_{K\times K\setminus D} |x-y|^2j_\gamma(x,dy)\fm_\gamma(dx)<\infty$ for any compact $K\subset\mathbb{R}^d$.
\end{itemize}
For (j.1), take $\varepsilon>0$ and an arbitrary compact set $K\subset\mathbb{R}^d$. Choose $r>1$ sufficiently large such that $K\subset B(r):=\{x\in \bR^d: |x|< r\}$. It follows from \eqref{EQ2UXP} that
\[
\begin{aligned}
\int_K j_\gamma(x,\bR^d\setminus U_\varepsilon(x))\fm_\gamma(dx)&=\int_K\int_{|x-y|>\varepsilon}|x-y|^{-(d+\alpha)}\psi_\gamma(y)\psi_\gamma(x)dydx\\
&\lesssim \int_K\int_{|x-y|>\varepsilon}|x-y|^{-(d+\alpha)}|x|^{\alpha-d}|y|^{\alpha-d}dydx.
\end{aligned}
\]
Denote $G_1:=\{y: |x-y|>\varepsilon,|y|>r\}$ and $G_2:=\{y:|x-y|>\varepsilon,|y|\leq r\}$. Then we have
\begin{equation}\label{EQ3KGX}
\begin{aligned}
\int_K\int_{G_1}|x-y|^{-(d+\alpha)}&|x|^{\alpha-d}|y|^{\alpha-d}dydx \\ &\leq 
\int_K\int_{G_1}\left(\frac{|y|}{|x-y|}\right )^{\alpha+d}|x|^{\alpha-d}|y|^{-2d}dydx.
\end{aligned}
\end{equation}
Take $x\in K$ and $y\in G_1$. Since $|y|>r$ and $|x|\leq \sup\{|z|:z\in K\}<r$, it follows that $|y|/|x-y|\leq r/(r-|x|)\leq C_{r,K}$ for a finite constant $C_{r,K}$. Hence the right hand side of \eqref{EQ3KGX} is not greater than
\[
C_{r,K}^{\alpha+d}\int_K|x|^{\alpha-d}dx\int_{|y|>r}|y|^{-2d}dy<\infty.
\]
In addition, 
\[
\int_K\int_{G_2}|x-y|^{-(d+\alpha)}|x|^{\alpha-d}|y|^{\alpha-d}dydx \leq  \varepsilon^{-(\alpha+d)}\int_K|x|^{\alpha-d}dx\int_{|y|\leq r}|y|^{\alpha-d}dy<\infty.
\]
Consequently, one can conclude (j.1). 
The second item (j.2) is obvious. For (j.3), we still take $r$ such that $K\subset B(r)$. When $\alpha+d\leq 2$, we have
\[
\begin{aligned}
	\int_{K\times K\setminus D} |x-y|^2j_\gamma(x,dy)\fm_\gamma(dx) &\leq \int_{K\times K\setminus D} |x-y|^{2-(\alpha+d)} |x|^{\alpha-d}|y|^{\alpha-d}dydx  \\
	&\leq (2r)^{2-(\alpha+d)}\int_{K\times K\setminus D}|x|^{\alpha-d}|y|^{\alpha-d}dydx<\infty. 
\end{aligned}\]
When $\alpha+d>2$, take a constant $\varepsilon>0$ and denote $K_1:=\{y:y\in K, |y-x|>\varepsilon\}$, $K_2:=\{y: y\in K, 0<|y-x|\leq \varepsilon\}$. Then
\[
\begin{aligned}
	\int_{K\times K\setminus D}& |x-y|^2j_\gamma(x,dy)\fm_\gamma(dx) \\
		&= \int_{K}\int_{K_1} |x-y|^2j_\gamma(x,dy)\fm_\gamma(dx)+\int_{K}\int_{K_2} |x-y|^2j_\gamma(x,dy)\fm_\gamma(dx).
\end{aligned}
\]
The first term is not greater than
\[
	\int_K\int_{K_1}|x-y|^{-(\alpha+d-2)}|x|^{\alpha-d}|y|^{\alpha-d}dydx\leq \varepsilon^{-(\alpha+d-2)}\left(\int_K|x|^{\alpha-d}dx\right)^2<\infty,
\]
and the second is not greater than
\[
\begin{aligned}
\int_K&\int_{K_2}|x-y|^{-(\alpha+d-2)}|x|^{\alpha-d}|y|^{\alpha-d}dydx\\
&\leq \left(\int_K\int_{K_2}\frac{|x|^{2(\alpha-d)}}{|x-y|^{\alpha+d-2}}dydx\right)^{1/2}\cdot\left(\int_K\int_{K_2}\frac{|y|^{2(\alpha-d)}}{|x-y|^{\alpha+d-2}}dydx\right)^{1/2}\\
&\leq \int_K|x|^{2(\alpha-d)}dx\int_{|y|\leq\varepsilon}|y|^{-(\alpha+d-2)}dy<\infty.
\end{aligned}
\]
Hence (j.3) is verified. 

\subsubsection{Step 2}\label{SEC312}
Note that (j.3) implies $C_c^\infty(\bR^d)\subset \sF^{(\gamma)}$ as well. Denote the $\sE^{(\gamma)}_1$-closure of $C_c^\infty(\bR^d)$ in $\sF^{(\gamma)}$ by $\bar\sF$. Then $(\sE^{(\gamma)}, \bar\sF)$ is also a Dirichlet form on $L^2(\bR^d,\fm_\gamma)$. Further let $\sA$ and $\bar\sA$ be the generators of  $(\sE^{(\gamma)}, \sF^{(\gamma)})$ and $(\sE^{(\gamma)}, \bar\sF)$ respectively. In this step, we show
\begin{equation}\label{EQ3CCRD}
\begin{aligned}
C_c^\infty&(\bR^d\setminus \{0\})\subset \cD(\sA)\cap \cD(\bar{\sA}), \\
\sA f(x)&=\bar{\sA}f(x) =\frac{c_{\alpha,d}}{\psi_\gamma(x)} \left(\text{p.v.}\int_{\bR^d} \frac{f(y)-f(x)}{|x-y|^{d+\alpha}}\psi_\gamma(y)dy\right) \\
& :=\frac{c_{\alpha,d}}{\psi_\gamma(x)} \left(\lim_{r\downarrow 0}\int_{y: |y-x|>r} \frac{f(y)-f(x)}{|x-y|^{d+\alpha}}\psi_\gamma(y)dy\right),\quad f\in C_c^\infty(\bR^d\setminus \{0\}),
\end{aligned}
\end{equation}
where the limit is in the sense of $L^2(\bR^d)$. To this end, we first show 
\begin{equation}\label{EQ3LFR}
	Lf(x):=\text{p.v.}\int_{\bR^d} \frac{f(y)-f(x)}{|x-y|^{d+\alpha}}\psi_\gamma(y)dy
\end{equation}
is well defined in $L^2(\bR^d)$ for all $f\in C_c^\infty(\bR^d\setminus \{0\})$. Fix such $f$ and take a compact set $K\subset \bR^d\setminus \{0\}$ such that $\text{supp}[f]\subset K$ and $\delta:=\inf\{|x-y|:x\in \text{supp}[f], y\in K^c\}>0$. On one hand, it follows from Lemma~\ref{LMB}~(1) and Minkovski's inequality that
\[
\begin{aligned}
	\left(\int_{K^c}Lf(x)^2dx\right)^{1/2} &\leq \left(\int_{K^c} \left(\int_{\text{supp}[f]} \frac{f(y)\psi_\gamma(y)}{|x-y|^{d+\alpha}}dy\right)^2dx\right)^{1/2}\\
	&\leq \int_{\text{supp}[f]} \left(\int_{K^c}\frac{dx}{|x-y|^{2(d+\alpha)}}\right)^{1/2}  f(y)\psi_\gamma(y)dy\\
	&\leq \|\psi_\gamma\|_K\int_{\bR^d} f(y)dy\left(\int_{x: |x-y|\geq \delta} \frac{dx}{|x-y|^{2(d+\alpha)}}\right)^{1/2} <\infty,
\end{aligned}\]
where $\|\psi_\gamma\|_K:=\sup\{|\psi_\gamma(x)|: x\in K\}<\infty$. 
On the other hand, 
\[
\begin{aligned}
	\int_{K}&Lf(x)^2dx \\&\leq \int_{K} \left(\int_{K} \frac{(f(y)-f(x))\psi_\gamma(y)}{|x-y|^{d+\alpha}}dy+\int_{K^c}\frac{-f(x)\psi_\gamma(y)}{|x-y|^{d+\alpha}}dy\right)^2dx\\
	&\lesssim \int_K\left(\int_{K} \frac{(f(y)-f(x))\psi_\gamma(y)}{|x-y|^{d+\alpha}}dy\right)^2dx+\int_K\left(\int_{K^c}\frac{\psi_\gamma(y)dy}{|x-y|^{d+\alpha}}\right)^2f(x)^2dx \\
	&=:J_1+J_2. 
\end{aligned}
\]
To estimate $J_1$, take $h\in C_c^\infty(\bR^d\setminus \{0\})$ with $h\equiv 1$ on $K$ and set $\tilde{\psi}:=\psi_\gamma\cdot h\in C_c^\infty(\bR^d\setminus \{0\})$ due to Lemma~\ref{LMB}~(1). Then we have
\[
J_1=\int_K\left(\int_{K} \frac{(f(y)-f(x))\tilde{\psi}(y)}{|x-y|^{d+\alpha}}dy\right)^2dx.
\]
Note that
\[
\begin{aligned}
	\int_K&\left(\int_{K} \frac{(f(y)-f(x))(\tilde{\psi}(y)-\tilde{\psi}(x))}{|x-y|^{d+\alpha}}dy\right)^2dx\\ 
	&\qquad \qquad \leq \|\nabla f\|_\infty^2 \|\nabla \tilde\psi\|_\infty^2\int_K  \left(\int_K \frac{dy}{|x-y|^{d+\alpha-2}}\right)^2dx<\infty 
\end{aligned}\]
and since $f\in \cD(\Delta^{\alpha/2})$ (see Definition~\ref{DEFA1}), 
\[
	\int_K\left(\int_{K} \frac{(f(y)-f(x))\tilde{\psi}(x)}{|x-y|^{d+\alpha}}dy\right)^2dx\leq \|\tilde{\psi}\|_K^2 \int_{K} \left(\int_K \frac{f(y)-f(x)}{|x-y|^{d+\alpha}} dy\right)^2dx<\infty. 
\]
Hence one can obtain $J_1<\infty$. 
For the second term $J_2$, it follows from $\psi_\gamma\in L^2(\bR^d)$ that
\[
	J_2\leq \|\psi_\gamma\|_{L^2(\bR^d)}^2\int_{\text{supp}[f]}f(x)^2dx\int_{y: |y-x|\geq \delta} \frac{dy}{|x-y|^{2(d+\alpha)}}<\infty. 
\]
Eventually we can conclude $Lf\in L^2(\bR^d)$. Secondly, fix $f\in C_c^\infty(\bR^d\setminus \{0\})$. For any $g\in \sF^{(\gamma)}$ or $g\in \bar{\sF}$, one can easily deduce by \eqref{EQ3LFR} and $g\cdot \psi_\gamma \in L^2(\bR^d)$ that
\[
	\sE^{(\gamma)}(f,g)=-c_{\alpha ,d}\int_{\bR^d}Lf(x) g(x)\psi_\gamma(x)dx=\left(-\frac{c_{\alpha,d}}{\psi_\gamma}Lf, g\right)_{L^2(\bR^d,\fm_\gamma)},
\]
which leads to \eqref{EQ3CCRD}.

\subsubsection{Step 3}\label{SEC313}

Define a self-adjoint operator $\cA$ on $L^2(\bR^d)$ as follows:
\begin{equation}\label{EQ3DAF}
\begin{aligned}
&\cD(\cA):= \{f\in L^2(\bR^d): f/\psi_\gamma \in \cD(\sA)\}, \\
&\cA f:= \psi_\gamma \cdot \sA\left(\frac{f}{\psi_\gamma}\right)+\lambda_\gamma f,\quad f\in \cD(\cA). 
\end{aligned}
\end{equation}
Clearly, $\cA$ is a self-adjoint operator on $L^2(\bR^d)$. In this step, we assert $\cA$ is an extension of $\Delta^{\alpha/2}$ restricted to $C_c^\infty(\bR^d\setminus \{0\})$, i.e. 
\begin{equation}\label{EQ3CCR}
	C_c^\infty(\bR^d\setminus \{0\})\subset \cD(\cA),\quad \cA f=\Delta^{\alpha/2}f,\quad \forall f\in C_c^\infty(\bR^d\setminus \{0\}),
\end{equation} 
and
\begin{equation}\label{EQ3PGD}
\psi_\gamma\in \cD(\cA),\quad 	\cA \psi_\gamma=\lambda_\gamma \psi_\gamma. 
\end{equation}
Indeed, $C_c^\infty(\bR^d\setminus \{0\})\subset \cD(\cA)$ is clear by \eqref{EQ3CCRD} and Lemma~\ref{LMB}~(1). Fix $f\in C_c^\infty(\bR^d\setminus \{0\})$ and take arbitrary $g\in C_c^\infty(\bR^d)$. A straightforward computation yields
\begin{equation}\label{EQ3FXG}
\begin{aligned}
\bigg( \frac{f(x)}{\psi_\gamma(x)}&- \frac{f(y)}{\psi_\gamma(y)}\bigg)(g(x)-g(y))\psi_\gamma(x)\psi_\gamma(y)
\\&=(f(x)-f(y))\cdot (g(x)\psi_\gamma(x))-(f(x)-f(y))\cdot (g(y)\psi_\gamma(y)) \\
 &\quad +(f(x)g(x)-f(y)g(y))\cdot \psi_\gamma(y)-(f(x)g(x)-f(y)g(y))\cdot \psi_\gamma(x).
\end{aligned}
\end{equation}
Recall that $L_I$ and $L_S$ are two equivalent expressions of $\Delta^{\alpha/2}$ as shown in Definition~\ref{DEFA1}. Since $f\in H^\alpha(\bR^d)=\cD(\Delta^{\alpha/2})$ and $g\cdot \psi_\gamma\in L^2(\bR^d)$,  we have 
\[
c_{\alpha,d}\int_{\bR^{d}}g(x)\psi_\gamma(x)dx\left(\text{p.v.}\int_{\bR^d}\frac{f(x)-f(y)}{|x-y|^{d+\alpha}}dy\right)=-\int_{\bR^d}L_If(x)g(x)\psi_\gamma(x)dx
\]
and 
\[
	c_{\alpha, d}\int_{\bR^{d}}g(y)\psi_\gamma(y)dy\left(\text{p.v.}\int_{\bR^d}\frac{f(x)-f(y)}{|x-y|^{d+\alpha}}dx\right)=\int_{\bR^d}L_If(x)g(x)\psi_\gamma(x)dx.
\]
Since $f\cdot g\in H^\alpha(\bR^d)$ and $\psi_\gamma \in L^2(\bR^d)$, it holds
\begin{equation}\label{EQ3CAD}
\begin{aligned}
c_{\alpha,d}\int_{\bR^d}\psi_\gamma(y)dy& \left(\text{p.v.}\int_{\bR^d}\frac{f(x)g(x)-f(y)g(y)}{|x-y|^{d+\alpha}}dx\right)\\
&=(\psi_\gamma, L_I(fg))_{L^2(\bR^d)}=(\psi_\gamma, L_S(fg))_{L^2(\bR^d)} \\
&=\lim_{t\downarrow 0}\left(\frac{1}{t}(p_t\ast \psi_\gamma -\psi_\gamma), fg  \right)_{L^2(\bR^d)}.
\end{aligned}\end{equation}
By virtue of $f\in C_c^\infty(\bR^d\setminus \{0\})$ and Lemma~\ref{LMB}~(3), the last term is equal to $\lambda_\gamma\int_{\bR^d} f(x)g(x)\psi_\gamma(x)dx$. 
Analogically,
\[
c_{\alpha,d}\int_{\bR^d}\psi_\gamma(x)dx\left(\text{p.v.}\int_{\bR^d}\frac{f(x)g(x)-f(y)g(y)}{|x-y|^{d+\alpha}}dy\right)=-\lambda_\gamma \int_{\bR^d}f(x)g(x)\psi_\gamma(x)dx. 
\]
Hence \eqref{EQ3FXG} tells us 
\begin{equation}\label{EQ3EGF}
	\sE^{(\gamma)}(f/\psi_\gamma, g)=-\int_{\bR^d} L_If(x)g(x)\psi_\gamma(x)dx+\lambda_\gamma \int_{\bR^d}f(x)g(x)\psi_\gamma(x)dx. 
\end{equation}
Since $f/\psi_\gamma \in \cD(\sA)$, we can obtain
\[
	\sA\left(f/\psi_\gamma\right)=(L_If)/\psi_\gamma-\lambda_\gamma (f/\psi_\gamma). 
\]
From the definition of $\cA$, \eqref{EQ3CCR} can be eventually concluded. 
On the other hand, one can easily find that $1\in \sF^{(\gamma)}$ and $\sE^{(\gamma)}(1,f)=0$ for all $f\in \sF^{(\gamma)}$. This implies $1\in \cD(\sA)$ and $\sA 1=0$. From the definition of $\cA$, we obtain \eqref{EQ3PGD}. Therefore, $\cA=\cA_\gamma$, i.e. the self-adjoint extension with parameter $\gamma$ of $\Delta^{\alpha/2}$ restricted to $C_c^\infty(\bR^d\setminus \{0\})$. Particularly, $\sA$ is identified with $\sA_\gamma$ in \eqref{EQ2DAG}.

\subsubsection{Step 4}\label{SEC314}
We can define an analogical self-adjoint operator $\bar{\cA}$ of $\cA$ on $L^2(\bR^d)$ by taking $\bar{\sA}$ in place of $\sA$ in \eqref{EQ3DAF}. Mimicking the proof of \eqref{EQ3CCR}, one can figure out that $\bar{\cA}$ is also an extension of $\Delta^{\alpha/2}$ restricted to $C_c^\infty(\bR^d\setminus \{0\})$. In addition,  take $\tau\in C_c^\infty(\bR^d)$ such that $0\leq \tau\leq 1$ and $\tau\equiv 1$ on $\{x: |x|\leq 1\}$. Set $\tau_n(x):=\tau(x/n)\in C_c^\infty(\bR^d)$. Then $\tau_n\rightarrow 1$ in $L^2(\bR^d, \fm_\gamma)$ as $n\uparrow \infty$ by the dominated convergence theorem. It follows from \eqref{EQ2UXP} that
\begin{equation}\label{EQ3EGT}
\begin{aligned}
	\sE^{(\gamma)}(\tau_n-1,\tau_n-1)&\leq \frac{c_{\alpha,d}\cdot c_{-\alpha,d}^2}{2}\iint \frac{\left(\tau(x/n)-\tau(y/n)\right)^2}{|x-y|^{d+\alpha}}|x|^{\alpha-d}|y|^{\alpha-d}dxdy \\
	&\lesssim \frac{1}{n^{d-\alpha}}\iint \frac{\left(\tau(x)-\tau(y)\right)^2}{|x-y|^{d+\alpha}}|x|^{\alpha-d}|y|^{\alpha-d}dxdy.
\end{aligned}
\end{equation} 
It is straightforward to verify that this integration is finite by mimicking the proof of (j.3) and thus $\sE^{(\gamma)}(\tau_n-1,\tau_n-1)\rightarrow 0$ as $n\rightarrow \infty$. Particularly, we can conclude $1\in \bar{\sF}$ and clearly, $\sE^{(\gamma)}(1,f)=0$ for all $f\in \bar{\sF}$. This indicates $1\in \cD(\bar{\sA})$ and $\bar{\sA}1=0$. From the definition of $\bar{\cA}$, we also have $\psi_\gamma\in \cD(\bar{\cA})$ and $\bar{\cA}\psi_\gamma=\lambda_\gamma \psi_\gamma$. Therefore, $\bar{\cA}=\cA_\gamma=\cA$ and hence $\bar{\sA}=\sA_\gamma=\sA$, which implies $\bar{\sF}=\sF^{(\gamma)}$. In other words, $(\sE^{(\gamma)},\sF^{(\gamma)})$ is regular on $L^2(\bR^d,\fm_\gamma)$ with a core $C_c^\infty(\bR^d)$. 

\subsubsection{Step 5}\label{SEC315}
Finally, we show $X^{(\gamma)}$ is associated with $(\sE^{(\gamma)},\sF^{(\gamma)})$. 
Note that the transition density of $X^{(\gamma)}$ with respect to the Lebesgue measure is (see \cite[(30)]{bib5})
\[
	q_\gamma(t,x,y)=\frac{\mathrm{e}^{-\lambda_\gamma t}p_\gamma(t,x,y)\psi_\gamma(y)}{\psi_\gamma(x)}, 
\]
where $p_\gamma$ is in \eqref{EQ1ZTX}. Clearly, its semigroup $Q^{(\gamma)}_t f:=\int_{\bR^d}q_\gamma(t,\cdot, y)f(y)dy$ is symmetric with respect to $\fm_\gamma$, i.e. $\int Q^{(\gamma)}_t f(x)g(x)\fm_\gamma(dx)=\int f(x) Q^{(\gamma)}_tg(x)\fm_\gamma(dx)$ for all $t\geq 0$ and suitable functions $f,g$. On the other hand, from \eqref{EQ2DAG} one can obtain that the semigroup of $(\sE^{(\gamma)},\sF^{(\gamma)})$ is
\[
	Q_t f=\mathrm{e}^{t\sA_\gamma}f=\frac{\mathrm{e}^{-\lambda_\gamma t}}{\psi_\gamma} \mathrm{e}^{t\cA_\gamma}(f\psi_\gamma)=\frac{\mathrm{e}^{-\lambda_\gamma t}}{\psi_\gamma}\int p_\gamma(t,\cdot, y)f(y)\psi_\gamma(y)dy
\]
for all $f\in L^2(\bR^d, \fm_\gamma)$. Hence $Q_t$ is identified with $Q^{(\gamma)}_t$ by a standard argument. That completes the proof. 

\subsection{Resolvent}
As shown in \S\ref{SEC315}, the semigroup of $(\sE^{(\gamma)},\sF^{(\gamma)})$ is 
\begin{equation}\label{EQ3QTG}
	Q_t^{(\gamma)} f=\frac{\mathrm{e}^{-\lambda_\gamma t}}{\psi_\gamma}P_t^{(\gamma)}(f\cdot \psi_\gamma),\quad  f\in L^2(\bR^d, \fm_\gamma), t\geq 0,
\end{equation}
where $P^{(\gamma)}_t$ is the semigroup associated with $\cA_\gamma$. Then its resolvent $R^{(\gamma)}_\lambda$ is
\[
	R^{(\gamma)}_\lambda f=\frac{1}{\psi_\gamma}U^{(\gamma)}_{\lambda_\gamma+\lambda}(f\cdot \psi_\gamma) \quad  f\in L^2(\bR^d, \fm_\gamma), \lambda> 0,
\]
where $U^{(\gamma)}$ is the resolvent of $\cA_\gamma$. From \cite[(20)]{bib5}, we conclude the following.

\begin{corollary}\label{COR31}
Let $U_\lambda$ be the resolvent of isotropic $\alpha$-stable process, i.e. for $\lambda>0$, $U_\lambda f= u_\lambda \ast f$ for $f\in L^2(\bR^d)$. Then the resolvent $R^{(\gamma)}$ associated with $(\sE^{(\gamma)},\sF^{(\gamma)})$ is expressed as follows: for $\lambda>0$ and $ f\in L^2(\bR^d, \fm_\gamma)$,
\begin{equation}\label{EQ3RGL}
	R^{(\gamma)}_\lambda f=\frac{1}{\psi_\gamma}U_{\lambda+\lambda_\gamma}(f\cdot\psi_\gamma)+\left(c^{(\gamma)}_\lambda \int_{\bR^d}f(x)\psi_\gamma(x)u_{\lambda+\lambda_\gamma}(x)dx\right)\cdot \frac{u_{\lambda+\lambda_\gamma}}{\psi_\gamma},
\end{equation}
where $c^{(\gamma)}_\lambda=\frac{1}{c(\alpha,d)(\lambda+\lambda_\gamma)^{\frac{d}{\alpha}-1}-\gamma}$ is a positive constant and $c(\alpha,d)$ is the constant in \eqref{EQ2LGG}. 
\end{corollary}
\begin{remark}
For the critical case $\gamma=0$, the analogical expression of the resolvent is still available, see \S\ref{SEC415}.
\end{remark}

\subsection{Global properties}

In this short subsection, we illustrate that $X^{(\gamma)}$ is an irreducible and recurrent (hence also conservative) Markov process by virtue of Theorem~\ref{THM21}.   Meanwhile, it is ergodic as explained in \eqref{EQ2TTP}.

\begin{proposition}\label{THM45}
The Dirichlet form $(\sE^{(\gamma)}, \sF^{(\gamma)})$ is irreducible and recurrent.
\end{proposition}
\begin{proof}
Note that $1\in\mathscr{F}^{(\gamma)}$ and $\mathscr{E}^{(\gamma)}(1,1)=0$. Then the recurrence of $(\mathscr{E}^{(\gamma)}, \mathscr{F}^{(\gamma)})$ follows from \cite[Theorem 2.1.8]{bib19}. To show the irreducibility, take $f\in \sF^{(\gamma)}$ with $\sE^{(\gamma)}(f,f)=0$. Since $\psi_\gamma(x)>0$ for all $x\in \bR^d\setminus \{0\}$, one can easily deduce that $f$ is a.e. constant. Eventually applying \cite[Theorem 2.1.10]{bib19}, we conclude that $(\mathscr{E}^{(\gamma)}, \mathscr{F}^{(\gamma)})$ is irreducible. That completes the proof.  
\end{proof}


\section{Critical state}\label{SEC4}

Now we turn to consider the case $\gamma=0$. The first task is to prove Theorem~\ref{THM22}. 

\subsection{Proof of Theorem~\ref{THM22}}
 We will also complete this proof in several steps. 

\subsubsection{Step 1}\label{SEC411}
 Mimicking \S\ref{SEC311}, one can also demonstrate that $(\sE^{(0)},\sF^{(0)})$ is a Dirichlet form on $L^2(\bR^d,\fm_0)$. In addition, $C_c^\infty(\bR^d)\subset \sF^{(0)}$. 

\subsubsection{Step 2}\label{SEC412}
In this step, we aim to show the denseness of $C_c^\infty(\bR^d)$ in $\sF^{(0)}$ relative to the $\sE^{(0)}_1$-norm. Since its generator is expected to correspond to a self-adjoint extension of $\Delta^{\alpha/2}$ restricted to $C_c^\infty(\bR^d\setminus\{0\})$ with no eigenfunctions ($\psi_0\notin L^2(\bR^d)$), the tactic of the proof in \S\ref{SEC314} is no longer available. Instead, we will prove it by a polishing technique appeared in e.g. \cite{bib18} as follows. 

Firstly,  we show the family of all bounded functions with compact support in $\sF^{(0)}$ is $\sE^{(0)}_1$-dense in $\sF^{(0)}$. Clearly, so is the family of all bounded functions in $\sF^{(0)}$. Fix a bounded $f\in \sF^{(0)}$. Take $\tau, \tau_n$ as in \S\ref{SEC314} and set $\eta_n:=1-\tau_n$, $f_n:=f\cdot \tau_n$. Then $f_n\in \sF^{(0)}$ is bounded with compact support. It suffices to show $\sE^{(0)}_1(f-f_n,f-f_n)\rightarrow 0$ as $n\rightarrow \infty$. Indeed, $\|f-f_n\|_{L^2(\bR^d,\fm_0)}\rightarrow 0$ as $n\rightarrow \infty$ by the dominated convergence theorem. In addition, 
\[
\begin{aligned}
\sE^{(0)}(f-f_n,f-f_n)&=\frac{c_{\alpha, d}}{2}\iint \left(f(x)\eta_n(x)-f(y)\eta_n(y)\right)^2\frac{\psi_0(x)\psi_0(y)}{|x-y|^{d+\alpha}}dxdy.
\end{aligned}
\]
Note that $$\left(f(x)\eta_n(x)-f(y)\eta_n(y)\right)^2\lesssim f(x)^2\left(\tau_n(x)-\tau_n(y)\right)^2+\eta_n(y)^2\left(f(x)-f(y)\right)^2.$$
Since $f$ is bounded, it follows that  
\[
\begin{aligned}
\iint f(x)^2&\left(\tau_n(x)-\tau_n(y)\right)^2\frac{\psi_0(x)\psi_0(y)}{|x-y|^{d+\alpha}}dxdy \\
&\leq \|f\|_\infty^2\iint \left(\tau_n(x)-\tau_n(y)\right)^2\frac{\psi_0(x)\psi_0(y)}{|x-y|^{d+\alpha}}dxdy \rightarrow 0
\end{aligned}\]
by mimicking \eqref{EQ3EGT}. By the dominated convergence theorem, one can also obtain
\[
\lim_{n\rightarrow\infty}\iint \eta_n(y)^2\left(f(x)-f(y)\right)^2\frac{\psi_0(x)\psi_0(y)}{|x-y|^{d+\alpha}}dxdy=0.
\]
Hence we can conclude $\sE^{(0)}(f-f_n,f-f_n)\rightarrow 0$ as $n\rightarrow \infty$.

Secondly, fix a bounded $f\in \sF^{(0)}$ with compact support and we will show that there is a sequence $\{f_n: n\geq 1\}\subset C_c^\infty(\bR^d)$ such that $\sE^{(0)}_1(f-f_n,f-f_n)\rightarrow 0$ as $n\rightarrow \infty$. To this end, take a radially symmetric, radially decreasing function $\rho\in C_{c}^{\infty}(\mathbb{R}^d)$ such that $\rho\geq 0$, $\text{supp}[\rho]\subset \{x:|x|<1\}$ and $\int_{\mathbb{R}^d}\rho(x)dx=1$. In other words, there exists a decreasing function $\hat{\rho}$ on $[0,\infty)$ such that $\rho(x)=\hat{\rho}(|x|)$. For every $\delta>0$, define $\rho_\delta(x):=\delta^{-d}\rho(x/\delta)$ and $f_\delta(x) :=\rho_\delta\ast f(x)=\int_{\mathbb{R}^d}\rho_\delta(x-y)f(y)dy$. Since $f$ is bounded with compact support, it follows that $f_\delta\in C_c^\infty(\bR^d)$. Clearly, $f_\delta\rightarrow f$ as $\delta\downarrow 0$ in $L^2(\bR^d, \fm_0)$. So it remains to show $\mathscr{E}^{(0)}(f-f_\delta, f-f_\delta)\rightarrow 0$ as $\delta\downarrow 0$. Fix an arbitrary constant $\varepsilon>0$. Note that
\begin{equation}\label{EQ4FFX}
	F_f(x,y):=\frac{f(x)-f(y)}{|x-y|^{\frac{d+\alpha}{2}}} \in L^2\left(\bR^{2d}, \psi_0(x)\psi_0(y)dxdy\right)=:H.
\end{equation}
Since $\psi_0(x)\psi_0(y)dxdy$ is a Radon measure on $\bR^{2d}$, one can take a function $g\in C_c(\bR^{2d})$ such that $\|g-F_f\|_H<\varepsilon$. For every function $h(x,y)\in H$, define
\[
	h\star \rho_\delta(x,y):=\int_{\mathbb{R}^d}h(x-z,y-z)\rho_\delta(z)dz. 
\]
This special convolution was frequently used in \cite{bib18}. Particularly by \cite[(6.5)]{bib18}, 
\begin{equation}\label{EQ4GGS}
	\|g-g\star\rho_\delta\|_{H}\rightarrow 0,\quad \delta\rightarrow 0.
\end{equation}
In addition, let $F_{f_\delta}$ be the function defined by \eqref{EQ4FFX} with $f_\delta$ in place of $f$. Then
\[
\|g\star\rho_\delta-F_{f_\delta}\|_{H}=\|g\star\rho_\delta-F_f\star\rho_\delta\|_{H}=\|(g-F_f)\star\rho_\delta\|_{H}
\lesssim \|g-F_f\|_{H}.
\]
The last inequality is due to \cite[Proposition 4.4]{bib18}. As a result, 
\[
\begin{aligned}
\mathscr{E}^{(0)}(f-f_\delta, f-f_\delta)^{\frac{1}{2}}&=\|F_f-F_{f_\delta}\|_H \\
 &\leq \|F_f-g\|_{H}+\|g-g\star\rho_\delta\|_{H}+\|g\star\rho_\delta-F_{f_\delta}\|_{H}\\
 &\lesssim 2\varepsilon+\|g-g\star\rho_\delta\|_{H}.
\end{aligned}
\]
Therefore we can conclude $\mathscr{E}^{(0)}(f-f_\delta, f-f_\delta)\rightarrow 0$ as $\delta\downarrow 0$ by \eqref{EQ4GGS}. 

\subsubsection{Step 3} Denote the generator of $(\sE^{(0)},\sF^{(0)})$ by $\sA$. We assert
\begin{equation}\label{EQ4CCRD}
\begin{aligned}
C_c^\infty&(\bR^d\setminus \{0\})\subset \cD(\sA), \\
\sA f(x)&=\frac{c_{\alpha,d}}{\psi_0(x)} \left(\text{p.v.}\int_{\bR^d} \frac{f(y)-f(x)}{|x-y|^{d+\alpha}}\psi_0(y)dy\right),\quad f\in C_c^\infty(\bR^d\setminus \{0\}).
\end{aligned}
\end{equation}
This claim can be verified by repeating \S\ref{SEC312} with $\gamma=0$ except for the estimate of $J_2$. Instead, take a constant $r>0$ such that $r/2>\sup\{|x|:x\in K\}$ and note that for $y\notin B(r)=\{z: |z|<r\}$ and $x\in K$, $|x-y|\geq |y|-|x|\geq |y|/2$. It follows that
\[
\int_K\left(\int_{K^c\cap B(r)^c}\frac{\psi_0(y)dy}{|x-y|^{d+\alpha}}\right)^2f(x)^2dx \lesssim \int_K f(x)^2dx \left(\int_{B(r)^c}\frac{\psi_0(y)dy}{|y|^{d+\alpha}}\right)^2<\infty,
\]
since $\psi_0(y)=c_{-\alpha,d}|y|^{\alpha-d}$. 
In addition,
\[
\int_K\left(\int_{K^c\cap B(r)}\frac{\psi_0(y)dy}{|x-y|^{d+\alpha}}\right)^2f(x)^2dx \leq \frac{\left(\int_{B(r)}\psi_0(y)dy\right)^2}{\delta^{2(d+\alpha)}}\int_K f(x)^2dx<\infty,
\]
due to the definition of $\delta$ in \S\ref{SEC312}. Hence $J_2<\infty$ and \eqref{EQ4CCRD} holds. 

\subsubsection{Step 4}

Next, define a self-adjoint operator $\cA$ on $L^2(\bR^d)$:
\begin{equation}\label{EQ4DAF}
\begin{aligned}
&\cD(\cA):= \{f\in L^2(\bR^d): f/\psi_0 \in \cD(\sA)\}, \\
&\cA f:= \psi_0 \cdot \sA\left(\frac{f}{\psi_0}\right),\quad f\in \cD(\cA). 
\end{aligned}
\end{equation}
We assert
\begin{equation}\label{EQ4CCR}
	C_c^\infty(\bR^d\setminus \{0\})\subset \cD(\cA),\quad \cA f=\Delta^{\alpha/2}f,\quad \forall f\in C_c^\infty(\bR^d\setminus \{0\}).
\end{equation}
To prove it, one can repeat the procedures from \eqref{EQ3FXG} to \eqref{EQ3EGF} with $\gamma=0$. However, the argument in \eqref{EQ3CAD} should be modified as follows (since $\psi_0\notin L^2(\bR^d)$). Note that $fg\in C_c^\infty(\bR^d\setminus \{0\})$. Take $r>0$ such that $\text{supp}[fg]\subset B(r)$. Then $\psi_0\cdot 1_{B(2r)}\in L^2(\bR^d)$ and it follows that
\begin{equation}\label{EQ4PBR}
\left(\psi_0\cdot 1_{B(2r)}, L_I(fg)\right)_{L^2(\bR^d)}=\lim_{t\downarrow 0}\left(\frac{1}{t}\left(p_t\ast (\psi_0\cdot 1_{B(2r)})-\psi_0\cdot 1_{B(2r)}\right), fg\right)_{L^2(\bR^d)}. 
\end{equation}
On the other hand, fix $y\notin B(2r)$ and then $|x-y|\geq |y|-|x|\geq |y|/2$ for all $x\in \text{supp}[fg]\subset B(r)$. Hence
\[
|L_I(fg)(y)|=\left|\text{p.v.}\int_{\bR^d}\frac{f(x)g(x)}{|x-y|^{d+\alpha}}dx\right|\lesssim \|fg\|_{L^1(\bR^d)} |y|^{-d-\alpha}. 
\]
Since $\psi_0(y)|y|^{-d-\alpha} 1_{B(2r)^c}(y)$ is integrable, one can obtain by the dominated convergence theorem that
\begin{equation}\label{EQ4BRC}
	\left|\int_{B(2r)^c} \psi_0(y)L_I(fg)(y)dy\right|<\infty. 
\end{equation}
From \cite[Lemma~3.4]{bib20}, we know $L_I(fg)(y)=L_S(fg)(y)$ for all $y\notin B(2r)^c$ and (see e.g. \cite[(S)]{bib20}), 
\[
	L_S(fg)(y)=\lim_{t\downarrow 0}\int_{\bR^d} (fg)(y+z)\frac{p_t(z)}{t}dz,\quad y\notin B(2r). 
\]
Actually this limit exists. Indeed, since $(fg)(y+z)\neq 0$ leads to $|y+z|<r$, it follows that $|z|\geq |y|-|y+z|>r$. Hence by \eqref{EQAPTX} and \eqref{EQAPCR}, it holds for $t<r^\alpha$, 
\[
\frac{p_t(z)}{t}=t^{-\frac{d+\alpha}{\alpha}}p_1\left(\frac{z}{t^{1/\alpha}}\right)\lesssim |z|^{-d-\alpha},
\]
which is integrable on $B(r)^c$. Mimicking \eqref{EQ4BRC}, it is straightforward to verify $$\int_{B(2r)^c}\psi_0(y)\int_{\bR^d}|(fg)(y+z)\|z|^{-d-\alpha}dz<\infty. $$ 
Then by the dominated convergence theorem and Fubini's theorem, one can obtain
\[
\begin{aligned}
	\int_{B(2r)^c} \psi_0(y)L_I(fg)(y)dy&=\int_{B(2r)^c} \psi_0(y)dy \left(\lim_{t\downarrow 0}\int_{\bR^d} (fg)(y+z)\frac{p_t(z)}{t}dz\right) \\\
	&=\lim_{t\downarrow 0}\int_{B(2r)^c} \psi_0(y)dy\int_{\bR^d} (fg)(y+z)\frac{p_t(z)}{t}dz \\
	&=\lim_{t\downarrow 0}\frac{1}{t}\int_{\bR^d}\left(p_t\ast (\psi_0\cdot 1_{B(2r)^c})\right)(z) (fg)(z)dz. 
\end{aligned}\]
From \eqref{EQ4PBR} and Lemma~\ref{LMB}~(3), we eventually conclude 
\[
	\int_{\bR^d}\psi_0(y)L_I(fg)(y)dy=\lim_{t\downarrow 0}\int_{\bR^d} \frac{1}{t}(p_t\ast\psi_0(y)-\psi_0(y))f(y)g(y)dy=0. 
\]

\subsubsection{Step 5}\label{SEC415}

Finally, it remains to prove $\cA$ given by \eqref{EQ4DAF} is exactly $\cA_0$, which leads to $\sA=\sA_0$. By the expression of the resolvent of $\cA_0$ (see e.g. \cite[(20)]{bib5}), it suffices to show the resolvent $R_\lambda$ associated with $(\sE^{(0)},\sF^{(0)})$ is identified with
\begin{equation}\label{EQ4RLF}
	R^{(0)}_\lambda f=\frac{1}{\psi_0}U_{\lambda}(f\cdot\psi_0)+\left(c^{(0)}_\lambda \int_{\bR^d}f(x)\psi_0(x)u_{\lambda}(x)dx\right)\cdot \frac{u_{\lambda}}{\psi_0}, \quad f\in L^2(\bR^d,\fm_0),
\end{equation}
where $c^{(0)}_\lambda=\frac{1}{c(\alpha,d)\lambda^{\frac{d}{\alpha}-1}}$. To this end, we apply a later result stated in Theorem~\ref{THM63}, i.e. take a sequence $\gamma_n\downarrow 0$ and then $(\sE^{(\gamma_n)},\sF^{(\gamma_n)})$ is convergent to $(\sE^{(0)},\sF^{(0)})$ in the sense of Mosco. The proof of it only relies on the expression of $(\sE^{(0)},\sF^{(0)})$ as we have proved in \S\ref{SEC411} and \S\ref{SEC412}. Recall that $R^{(\gamma_n)}$ denotes the resolvent of $(\sE^{(\gamma_n)},\sF^{(\gamma_n)})$ and is expressed in Corollary~\ref{COR31}. Particularly, Mosco convergence implies $R^{(\gamma_n)}_\lambda$ strongly converges to $R_\lambda$ in the sense of Definition~\ref{DEFB3}. By Lemma~\ref{LMB4}~(4), this  leads to 
\[
	\psi_{\gamma_n}\cdot \left(R^{(\gamma_n)}_\lambda f\right)\rightarrow \psi_0\cdot (R_\lambda f),\quad \text{ in }L^2(\bR^d) \text{ as } n\rightarrow \infty
\]
for all $f\in C_c^\infty(\bR^d)$. From \eqref{EQ3RGL}, one can easily find that $\psi_{\gamma_n}\cdot \left(R^{(\gamma_n)}_\lambda f\right)$ converges in $L^2(\bR^d)$ to 
\[
	U_\lambda(f\cdot \psi_0)+\left(c^{(0)}_\lambda \int_{\bR^d}f(x)\psi_0(x)u_{\lambda}(x)dx\right)\cdot u_{\lambda}=\psi_0\cdot (R^{(0)}_\lambda f). 
\]
Therefore $R_\lambda f=R^{(0)}_\lambda f$ for all $f\in C_c^\infty(\bR^d)$. By a standard argument, we can conclude $R_\lambda$ and $R^{(0)}_\lambda$ are identified. That completes the proof. 

\subsection{Global properties}

In this short subsection, we illustrate that $X^{(0)}$ is also irreducible and recurrent. 

\begin{proposition}\label{PRO41}
The Dirichlet form $(\sE^{(0)},\sF^{(0)})$ is irreducible and recurrent.
\end{proposition}
\begin{proof}
Take $\tau, \tau_n$ as in \S\ref{SEC314}. Note that $\tau_n\uparrow 1$ and $\sE^{(0)}(\tau_n,\tau_n)\rightarrow 0$ by mimicking \eqref{EQ3EGT}. Then it follows from \cite[Theorem 2.1.8]{bib19} that $(\sE^{(0)},\sF^{(0)})$ is recurrent. 

To show the irreducibility, suppose $f_n\in \sF^{(0)}$ such that $\lim_{n\rightarrow \infty}\sE^{(0)}(f_n,f_n)=0$ and $f(x):=\lim_{n\rightarrow \infty}f_n(x)$ exists for a.e. $x\in \bR^d$. We use the same notations as in \eqref{EQ4FFX}. Then $F_{f_n}\in H$ and $\|F_{f_n}\|_H\rightarrow 0$. This leads to $F_{f_{n_k}}\rightarrow 0$, $dxdy$-a.e. as $k\rightarrow \infty$ for a suitable subsequence $\{f_{n_k}:k\geq 1\}\subset \{f_n: n\geq 1\}$.  On the other hand, 
\[
	\lim_{k\rightarrow \infty}F_{f_{n_k}}(x,y)=\frac{f(x)-f(y)}{|x-y|^{(d+\alpha)/2}}
\]
for a.e. $(x,y)\in \bR^{2d}$. Hence we can conclude $f$ is a.e. constant. By applying \cite[Theorem 5.2.16]{bib19}, $(\mathscr{E}^{(0)}, \mathscr{F}^{(0)})$ is irreducible. That completes the proof.
\end{proof}

\section{Alternative characterization via $h$-transform}\label{SEC5}

In this section we reconsider globular states or critical state by means of so-called $h$-transform. Fix $\gamma\geq 0$. Recall that $W^\alpha$ is the isotropic $\alpha$-stable process on $\bR^d$ with $\frac{d}{2}<\alpha<d\wedge 2$. We use the notation $P_t$ to stand for the probability transition semigroup of $W^\alpha$ as well as the $L^2$-semigroup associated with \eqref{EQ2DGH} if no confusions cause. Clearly, $\psi_\gamma=u_{\lambda_\gamma}$ is $\lambda_\gamma$-excessive relative to $(P_t)$, i.e. 
\[
	\mathrm{e}^{-\lambda_\gamma t}P_t\psi_\gamma\leq \psi_\gamma, \quad \lim_{t\downarrow 0}\mathrm{e}^{-\lambda_\gamma t}P_t\psi_\gamma=\psi_\gamma.
\]
Following e.g. \cite[Chapter 11]{bib21}, one can derive a nice Markov process on $E_h:=\{x: 0<h(x)<\infty\}$ by virtue of well-known $h$-transform with $h:=\psi_\gamma$. More precisely, set
\begin{equation}\label{EQ5HPG}
	{}_hP^{(\gamma)}_t(x,dy):=\left\lbrace
\begin{aligned}
& \mathrm{e}^{-\lambda_\gamma t}\frac{\psi_\gamma(y)}{\psi_\gamma(x)}P_t(x,dy),\quad x\in E_h=\bR^d\setminus \{0\},\\
&0,\qquad \qquad \quad \quad\quad\quad\;\;\;\;\; x=0. 
\end{aligned}
\right. 
\end{equation}
Then $({}_hP^{(\gamma)}_t)$ is a sub-Markov semigroup and generates a Markov process, denoted by ${}_hW^{\alpha,(\gamma)}$, on $\bR^d\setminus \{0\}$ as shown in e.g. \cite[Theorem~11.9]{bib21}. 

To phrase the main result of this section, we prepare two notions. Let $E$ be a locally compact separable metric space and $\fm$ be a positive Radon measure on it. 
The first one is the so-called \emph{part process}; see \cite[\S4.4]{bib11}. Let $(\sE,\sF)$ be a Dirichlet form on $L^2(E,\fm)$ associated with a Markov process $X$  and $F\subset E$ be a closed set of positive capacity relative to $(\sE,\sF)$. Then the part process $X^G$ of $X$ on $G:=E\setminus F$  is obtained by killing $X$ once upon leaving $G$. In other words, 
\[
	X^G_t=\left\lbrace
	\begin{aligned}
	& X_t,\quad t<\sigma_F:=\{s>0: X_s=F\}, \\
	& \partial, \quad\;\; t\geq \sigma_F,
	\end{aligned}
	\right. 
\]
where $\partial$ is the trap of $X^G$. Note that $X^G$ is associated with the \emph{part Dirichlet form} $(\sE^G,\sF^G)$ of $(\sE,\sF)$ on $G$: 
\begin{equation}\label{EQ5FGF}
\begin{aligned}
	&\sF^G=\{f\in \sF: \tilde{f}=0,\; \sE\text{-q.e. on }F\}, \\
	&\sE^G(f,g)=\sE(f,g),\quad f,g\in \sF^G,
\end{aligned}
\end{equation}
where $\tilde{f}$ stands for the quasi-continuous version of $f$. The second is the one-point reflection of a Markov process studied in \cite{bib22}; see also \cite[\S7.5]{bib19}. Let $a\in E$ be a non-isolated point with $\fm(\{a\})=0$ and $X^0$ be an $\fm$-symmetric Borel standard process on $E_0:=E\setminus \{a\}$ with no killing inside. Then a right process $X$ on $E$ is called a \emph{one-point reflection} of $X^0$ (at $a$) if $X$ is $\fm$-symmetric and of no killing on $\{a\}$, and the part process of $X$ on $E_0$ is $X^0$. 

\begin{theorem}\label{THM44}
Fix $\gamma\geq 0$ and let $X^{(\gamma)}$ and $(\sE^{(\gamma)},\sF^{(\gamma)})$ be in Theorem~\ref{THM21} or Theorem~\ref{THM22}. Then $\{0\}$ is of positive capacity relative to $(\sE^{(\gamma)},\sF^{(\gamma)})$. Furthermore, the following hold:
\begin{itemize}
\item[(1)] ${}_hW^{\alpha,(\gamma)}$ is identified with the part process of $X^{(\gamma)}$ on $\bR^d\setminus \{0\}$; 
\item[(2)] $X^{(\gamma)}$ is the unique (in law) one-point reflection of ${}_hW^{\alpha,(\gamma)}$ at $0$. 
\end{itemize}
\end{theorem}
\begin{proof}
Denote the 1-capacity relative to $(\sE^{(\gamma)},\sF^{(\gamma)})$ by $\text{Cap}^{(\gamma)}$ (see \cite[\S2.1]{bib11}). Since $\psi_\gamma\leq \psi_0$ for $\gamma>0$, it follows from the definition of $1$-capacities that $\text{Cap}^{(\gamma)}(A)\leq\text{Cap}^{(0)}(A)$ for any Borel set $A\subset\mathbb{R}^d$. Hence we only need to show $\text{Cap}^{(\gamma)}(\{0\})>0$ for $\gamma>0$. Argue with contradiction and suppose $\text{Cap}^{(\gamma)}(\{0\})=0$ for some $\gamma>0$. Then the part process of $X^{(\gamma)}$ on $\bR^d\setminus\{0\}$ coincides with $X^{(\gamma)}$ and particularly, it follows from \cite[Theorem~4.4.3]{bib11} that $C_c^\infty(\bR^d\setminus \{0\})$ is also a core of $(\sE^{(\gamma)},\sF^{(\gamma)})$. By \eqref{EQ3EGF}, one can easily obtain that for any $f,g\in C_c^\infty(\bR^d\setminus\{0\})$, 
\begin{equation}\label{EQ5EGF}
\sE^{(\gamma)}\left(\frac{f}{\psi_\gamma},\frac{g}{\psi_\gamma}\right)=\sG_{\lambda_\gamma}(f,g). 
\end{equation}
Note that $C_c^\infty(\bR^d\setminus \{0\})$ is a core of $(\sG, \cD(\sG))$ due to $\alpha<d$. This implies that
\[
	f\mapsto \frac{f}{\psi_\gamma}
\]
is an isomorphism between $\cD(\sG)$ with the norm $\|\cdot\|_{\sG_{\lambda_\gamma+1}}$ and $\sF^{(\gamma)}$ with the norm $\|\cdot\|_{\sE^{(\gamma)}_1}$. Particularly, the operator $\cA$ defined by \eqref{EQ3DAF} must be identified with $\Delta^{\alpha/2}$. This leads to contradiction, because we have shown $\cA=\cA_\gamma\neq \Delta^{\alpha/2}$ in \S\ref{SEC313}. 

To prove the first assertion, it is straightforward to verify that $({}_hP^{(\gamma)}_t)$ is symmetric with respect to $\fm_\gamma(dx)=\psi_\gamma(x)^2dx$ and then associated with the Dirichlet form (see \cite[(1.3.17)]{bib11})
\[
\begin{aligned}
\sF&=\{f\in L^2(\bR^d,\fm_\gamma): \sE(f,f)<\infty\}, \\
\sE(f,g)&=\lim_{t\downarrow 0}\frac{1}{t}\int_{\bR^d} \left(f(x)-{}_hP^{(\gamma)}_tf(x)\right)g(x)\fm_\gamma(dx),\quad f,g\in \sF. 
\end{aligned}
\]
One can easily deduce that for any $f\in L^2(\bR^d,\fm_\gamma)$, 
\[
	\sE(f,f)=\lim_{t\downarrow 0}\frac{1}{t}\int_{\bR^d} \left(f(x)\psi_\gamma(x)-\mathrm{e}^{-\lambda_\gamma t}P_t(f\psi_\gamma)(x)\right)(f\psi_\gamma)(x)dx=\sG_{\lambda_\gamma}(f\psi_\gamma, f\psi_\gamma). 
\]
This leads to
\begin{equation}\label{EQ5FFF}
	\sF=\{f: f\psi_\gamma \in \cD(\sG)\},\quad \sE(f,f)=\sG_{\lambda_\gamma}(f\psi_\gamma, f\psi_\gamma),\quad f\in \sF. 
\end{equation}
Since $C_c^\infty(\bR^d\setminus \{0\})$ is a core of $(\sG,\cD(\sG))$ and $\psi_\gamma \in C^\infty(\bR^d\setminus \{0\})$ is positive, we can conclude that $C_c^\infty(\bR^d\setminus \{0\})$ is also a core of $(\sE,\sF)$. 
On the other hand, the part process $X^{(\gamma),0}$ of $X^{(\gamma)}$ on $\bR^d\setminus \{0\}$ is associated with the Dirichlet form $(\sE^{(\gamma),0},\sF^{(\gamma),0})$ given by \eqref{EQ5FGF} with $(\sE,\sF)=(\sE^{(\gamma)},\sF^{(\gamma)})$ and $G=\bR^d\setminus \{0\}$. Particularly, $C_c^\infty(\bR^d\setminus \{0\})$ is also a core of $(\sE^{(\gamma),0},\sF^{(\gamma),0})$ by \cite[Theorem~4.4.3]{bib11}. Mimicking \eqref{EQ5EGF} and applying \eqref{EQ5FFF}, one can obtain that for any $f\in C_c^\infty(\bR^d\setminus\{0\})$,
\[
	\sE^{(\gamma),0}(f,f)=\sE^{(\gamma)}(f,f)=\sG_{\lambda_\gamma}(f\psi_\gamma, f\psi_\gamma)=\sE(f,f),
\]
which implies $(\sE^{(\gamma),0},\sF^{(\gamma),0})=(\sE,\sF)$. Therefore, ${}_hW^{\alpha,(\gamma)}$ is equivalent to the part process of $X^{(\gamma)}$ on $\bR^d\setminus \{0\}$.

Finally we prove the second assertion. Clearly, $X^{(\gamma)}$ is a one-point reflection of ${}_hW^{\alpha,(\gamma)}$ by the first assertion. Note that for $\sE^{(\gamma)}$-q.e. $x\neq 0$,
\begin{equation}\label{EQ5HPG2}
	{}_h\mathbf{P}_{\gamma}^x(\zeta_h<\infty, {}_hW^{\alpha,(\gamma)}_{\zeta_h-}=0)={}_h\mathbf{P}_{\gamma}^x(\zeta_h<\infty)=\mathbf{P}^x_\gamma(\sigma_0<\infty)=1,
\end{equation}
where ${}_h\mathbf{P}_{\gamma}^x$ is the probability measure of ${}_hW^{\alpha,(\gamma)}$ starting from $x$, $\zeta_h$ is its life time and $\sigma_0=\inf\{t>0: X^{(\gamma)}_t=0\}$. The first equality is due to the conservativeness of $X^{(\gamma)}$ and that ${}_hW^{\alpha,(\gamma)}=X^{(\gamma),0}$ has no killing inside, and the last equality is already mentioned in \eqref{EQ2PXG}. 
Applying  \cite[Theorem~7.5.4]{bib19}, we can eventually conclude the uniqueness of one-point reflections. That completes the proof.
\end{proof}
\begin{remark}
At a heuristic level, $\text{Cap}^{(\gamma)}(\{0\})>0$ is a reflection of the fact that $\cA_\gamma$ has infinite potential at $0$ as we can see in \eqref{EQ1DAB}. The analogical result for the three-dimensional Brownian case, i.e. $\alpha=2$ and $d=3$, has been obtained in \cite{bib1}. 
It is also worth noting that for any $x\neq 0$, $\{x\}$ is of zero capacity relative to $\sG$ as well as $\sE^{(\gamma)}$ due to $\alpha<d$ and \eqref{EQ5FFF}. 
\end{remark}

With the help of Theorem~\ref{THM44}, we summarize an alternative characterization of the polymer model based on $\alpha$-stable process in Figure~\ref{FIG1}. The $h$-transform from $\Delta^{\alpha/2}$ to ${}_h W^{\alpha,(\gamma)}$ is reversible. Indeed, one can operate a similar $h$-transform with $h=1/\psi_\gamma$ on ${}_hW^{\alpha,(\gamma)}$ to regain the $\alpha$-stable process; see \eqref{EQ5FFF}. The transformation \eqref{EQ3DAF} or \eqref{EQ4DAF} enjoys a same form as $h$-transform. However, $P^{(\gamma)}_t$ is not Markovian (although $P_t^{(\gamma)}\psi_\gamma=\mathrm{e}^{\lambda_\gamma t}\psi_\gamma$ by \eqref{EQ3QTG}) and $1/\psi_\gamma$ is not excessive relative to $Q^{(\gamma)}_t$ either. As mentioned before, \eqref{EQ1DAB} is a heuristic expression of the informal perturbation of $\Delta^{\alpha/2}$ induced by a singular potential function $\beta_\gamma\cdot \delta_0$. From Figure~\ref{FIG1}, we figure out a rigorous probabilistic interpretation for this perturbation: it may be understood as one-point reflection at $0$ under certain $h$-transform. 

We present a corollary to illustrate further properties of $X^{(\gamma)}$ as well as its Dirichlet form $(\sE^{(\gamma)}, \sF^{(\gamma)})$ by means of one-point reflection.

\begin{corollary}\label{COR53}
Fix $\gamma\geq 0$. The following hold:
\begin{itemize}
\item[(1)] $0$ is regular for itself with respect to $X^{(\gamma)}$, i.e. $\bfP_\gamma^0(\sigma_0=0)=1$; 
\item[(2)] For any $\lambda>0$, $w_\lambda(x):=\mathbf{E}^x_{\gamma}(\mathrm{e}^{-\lambda \sigma_0}; \sigma_0<\infty)$ is identified with
\[
	w_{\lambda, \lambda_\gamma}(x):=\left\lbrace
		\begin{aligned}
		&\frac{u_{\lambda+\lambda_\gamma}(x)}{\psi_\gamma(x)},\quad x\neq 0; \\
		&1,\qquad\qquad\;\;\, x=0.
		\end{aligned}
	\right.
\] 
More precisely, $w_\lambda(x)=w_{\lambda,\lambda_\gamma}(x)$ for $\sE^{(\gamma)}$-q.e. $x$. 
\item[(3)] $X^{(\gamma)}$ admits no jump to or from $\{0\}$: for $\sE^{(\gamma)}$-q.e. $x\in \bR^d$,
\begin{equation}\label{EQ5PGX}
	\bfP_\gamma^x(X^{(\gamma)}_{t-}\in \bR^d\setminus \{0\}, X^{(\gamma)}_t=0, \text{ or }X^{(\gamma)}_{t-}=0, X^{(\gamma)}_t\in \bR^d\setminus \{0\};\; \exists t>0)=0. 
\end{equation}
\item[(4)] Let $(\sE^{(\gamma),0},\sF^{(\gamma),0})$ be the Dirichlet form associated with ${}_hW^{\alpha,(\gamma)}$ and fix $\lambda>0$. Then it holds 
\[
	\sF^{(\gamma)}=\sF^{(\gamma),0}\oplus w_\lambda:=\{c_1 f+c_2w_\lambda: f\in \sF^{(\gamma),0}, c_1,c_2\in \bR\}.  
\]
Particularly, $C_c^\infty(\bR^d\setminus\{0\})\oplus w_\lambda=\{c_1 f+c_2w_\lambda: f\in C_c^\infty(\bR^d\setminus\{0\}), c_1,c_2\in \bR\}$ is $\sE^{(\gamma)}_1$-dense in $\sF^{(\gamma)}$.
\end{itemize}
\end{corollary}
\begin{proof}
The first and fourth assertions are consequences of \cite[Theorem~7.5.4]{bib19}. By comparing \eqref{EQ3RGL} or \eqref{EQ4RLF} with \cite[(7.5.6)]{bib19}, a straightforward computation yields $w_\lambda=w_{\lambda,\lambda_\gamma}$, $\fm_\gamma$-a.e. Since $w_{\lambda,\lambda_\gamma}$ is continuous by Lemma~\ref{LMB}~(2) and $w_\lambda$ is $\sE^{(\gamma)}$-quasi-continuous, it follows that  $w_\lambda(x)=w_{\lambda,\lambda_\gamma}(x)$ for $\sE^{(\gamma)}$-q.e. $x$. 

To show the third assertion, we shall apply \cite[Theorem~7.5.6]{bib19} and so it suffices to verify the conditions \textbf{(A.2) (A.3)} and \textbf{(A.4)} there. For any $\lambda>0$, it follows from the second assertion that
\[
	\int w_\lambda(x)\fm_\gamma(dx)=\int u_{\lambda+\lambda_\gamma}(x)u_{\lambda_\gamma}(x)dx<\infty. 
\]
Hence \textbf{(A.2)} holds. Note that $\varphi(x):=\mathbf{P}^x_\gamma(\sigma_0<\infty)\equiv 1$. Denote the resolvent of ${}_hW^{\alpha,(\gamma)}$ by $R^{(\gamma),0}_\lambda$. Then from $\psi_\gamma=u_{\lambda_\gamma}$, \eqref{EQ5HPG} and the resolvent equation, we obtain
\[
R^{(\gamma),0}_1\varphi=\int_0^\infty \mathrm{e}^{-t}{}_hP_t^{(\gamma)}\varphi dt=\frac{u_{\lambda_\gamma+1}\ast \psi_\gamma}{\psi_\gamma}=\frac{u_{\lambda_\gamma}-u_{\lambda_\gamma+1}}{u_{\lambda_\gamma}}=1-w_{1,\lambda_\gamma}. 
\]
Then it is easy to conclude from \eqref{EQAULXA} that for any compact set $K\subset \bR^d\setminus \{0\}$, 
\[
	\inf_{x\in K}R^{(\gamma),0}_1\varphi(x)=1-\sup_{x\in K}w_{1,\lambda_\gamma}(x)>0,
\] 
which leads to \textbf{(A.3)}. Note that the jumping measure of ${}_hW^{\alpha,(\gamma)}$ is $$J_0(dxdy)=\frac{c_{\alpha,d}}{2}\frac{\psi_\gamma(x)\psi_\gamma(y)}{|x-y|^{d+\alpha}}dxdy.$$ 
Fix $r>0$. For $x\in B(r), y\notin B(2r)$, it holds $|x-y|\geq |y|-|x|\geq |y|/2$. Thus
\[
\begin{aligned}
J_0(B(r)\times B(2r)^c)&\lesssim \int_{B(r)}|x|^{\alpha-d}dx\int_{B(2r)^c}\frac{|y|^{\alpha-d}}{|x-y|^{d+\alpha}}dy \\
&\lesssim \int_{B(r)}|x|^{\alpha-d}dx\int_{B(2r)^c}|y|^{-2d}dy<\infty. 
\end{aligned}
\] 
Consequently, \textbf{(A.4)} holds. That completes the proof. 
\end{proof}
\begin{remark}
\begin{itemize}
\item[(1)] When $d\geq 2$, \eqref{EQ5HPG2}, \eqref{EQ5PGX} and $w_\lambda(x)=w_{\lambda,\lambda_\gamma}(x)$ hold for all $x\in \bR^d$ if we replace $X^{(\gamma)}$ by a suitable equivalent version of it. Indeed, since $\psi_\gamma$ is a radial function, one can easily find that $X^{(\gamma)}$ is rotationally invariant. Consequently, all terms in \eqref{EQ5HPG2}, \eqref{EQ5PGX} and the second assertion of Corollary~\ref{COR53} depend on $|x|$ only. For any $r>0$, $\{x: |x|=r\}$ is of positive capacity relative to $W^\alpha$ due to $\alpha>d/2\geq 1$ (see e.g. \cite[Remark~2.2]{bib23}). Then it follows from \eqref{EQ5FFF} that $\{x: |x|=r\}$ is also of positive capacity relative to $\sE^{(\gamma)}$. Recall that $\{0\}$ is of positive capacity as shown in Theorem~\ref{THM44}. As a result, we can build an equivalent version $\tilde{X}$ of $X^{(\gamma)}$ by virtue of rotation invariance, so that these equalities hold for all $x$ relative to $\tilde{X}$. 
\item[(2)] \eqref{EQ5PGX} tells us the trajectories of $X^{(\gamma)}$ are not only c\`adl\`ag but also continuous at the moments $t$ when $X^{(\gamma)}_t=0$, although its associated Dirichlet form contains no diffusion part. 
\end{itemize}
\end{remark}

The first part of \eqref{EQ5PGX}, i.e. $X^{(\gamma)}$ admits no jump to $\{0\}$, can be verified by means of so-called L\'evy system (see e.g. \cite[A.3]{bib11}) directly. Note that the L\'evy system $(N,H)$ of $X^{(\gamma)}$ may be taken to be
\[
	N(x,dy)=c_{\alpha,d}\frac{\psi_\gamma(y)}{\psi_\gamma(x)}\frac{dy}{|x-y|^{d+\alpha}},\quad H_t=t. 
\]
Then from \cite[(A.3.23)]{bib11}, we obtain
\[
\begin{aligned}
	\mathbf{P}^x_\gamma(X^{(\gamma)}_{t-}&\in \bR^d\setminus \{0\}, X^{(\gamma)}_t=0;\; \exists t>0) \\ 
	&\leq \mathbf{E}^x_\gamma\left( \sum_{t> 0}1_{(\bR^d\setminus \{0\}) \times \{0\}}(X^{(\gamma)}_{t-},X^{(\gamma)}_t)\right)\\&=\mathbf{E}^x_\gamma \int_0^\infty dt\int_{\bR^d} 1_{(\bR^d\setminus \{0\}) \times \{0\}}(X^{(\gamma)}_t,y)N(X^{(\gamma)}_t,dy)=0. 
\end{aligned}\]
However, the other part of \eqref{EQ5PGX} was led in \cite[Theorem~7.5.6]{bib19} by a classical probabilistic construction of $X^{(\gamma)}$ initiated by It\^o. Roughly speaking, let $\{\nu_t:t>0\}$ be the unique ${}_hW^{\alpha,(\gamma)}$-\emph{entrance law}, i.e. $\nu_t$ is a $\sigma$-finite measure on $\bR^d\setminus \{0\}$ and $\nu_s\cdot  {}_hP^{(\gamma)}_t=\nu_{t+s}$ for every $t,s>0$, such that
\[
	\int_0^\infty \nu_tdt=\fm_\gamma. 
\]
This entrance law determines a so-called \emph{excursion measure} $\mathbf{n}$ on the space $W$ of c\`adl\`ag paths $w$ in $\bR^d\setminus \{0\}$ defined on a time interval $(0,\zeta(w))$ with $w(0+)=0$ and $w(\zeta-)\in \{0,\partial\}$ (the trap $\partial$ can be taken to be $\infty$ in the current case). Then a Poission point process $\mathbf{p}=\{\mathbf{p}_t: t\geq 0\}$ taking values in $W$ with characteristic $\mathbf{n}$ can be constructed on a suitable probability measure space. By piecing together the excursions $\mathbf{p}$ until the first non-returning excursion (i.e. $w(\zeta-)=\partial$), we create a path $\omega^0$ starting  at $0$. Then $X^{(\gamma)}$ can be eventually constructed by joining the path of ${}_hW^{\alpha,(\gamma)}$ to $\omega^0$. The details of this construction are referred to e.g. \cite[Theorem~7.5.6]{bib19}. Note that the path $\omega^0$ is continuous at the moments $t$ when $\omega^0(t)=0$, as indicates that $X^{(\gamma)}$ admits no jump from $\{0\}$ to $\bR^d\setminus \{0\}$.

\section{Near the critical point}\label{SEC6}

As mentioned in \S\ref{SEC1}, the behaviour of polymer near the critical point is measured by the parameter $\lambda_0(\gamma):=\sup\sigma(\cA_\gamma)=\lambda_\gamma$ given by \eqref{EQ2LGG} in \cite{bib5}. Note that $\lim_{\gamma\downarrow \gamma_{cr}} \lambda_0(\gamma)=0=\lambda_0(\gamma_{cr})$ and the rate of convergence is equal to
\[
	\lim_{\gamma\downarrow \gamma_{cr}}\frac{\log \lambda_0(\gamma)}{\log \gamma}=\frac{\alpha}{d-\alpha},
\]
which only depends on $d$ and $\alpha$. This fact demonstrates so-called universality of critical phenomenon as well.  In this section, we will describe the critical behaviour from probabilistic viewpoint by showing that the process $X^{(\gamma)}$ is convergent to $X^{(0)}$ as $\gamma\downarrow 0$ in a certain meaning. 

Fix a sequence $\gamma_n\downarrow 0$ and for convenience's sake, denote 
\[
	X^n:=X^{(\gamma_n)},\quad \bfP^x_n:=\bfP^x_{\gamma_n},\quad X:=X^{(0)},\quad \bfP^x:=\bfP^x_{0}.
\]
Take a non-negative function $\phi$ on $\bR^d$ such that 
\begin{equation}\label{EQ6RDP}
	\phi/\psi_{\gamma_1}\in L^2(\bR^d),\quad \int_{\bR^d}\phi(x)dx=1. 
\end{equation}
Then $\bfP^\phi_n(\cdot):=\int_{\bR^d}\bfP^x_n(\cdot)\phi(x)dx$ and $\bfP^\phi(\cdot):=\int_{\bR^d}\bfP^x(\cdot)\phi(x)dx$ define probability measures on $\Omega=D([0,\infty),\bR^d)$. The main result of this section is stated as follows.

\begin{theorem}\label{THM61}
Let $\phi$ be a non-negative function satisfying \eqref{EQ6RDP}. Then $X^n$ is weakly convergent to $X$  under the initial distribution $\phi(x)dx$ as $n\rightarrow \infty$. More precisely, for any bounded continuous function $f$ on $\Omega$ endowed with the Skorohod topology,
\begin{equation}\label{EQ6NOF}
\lim_{n\rightarrow \infty} \int_{\Omega}f(\omega) \bfP^\phi_n(d\omega)=\int_{\Omega}f(\omega) \bfP^\phi(d\omega).
\end{equation}
\end{theorem}
\begin{remark}\label{RM62}
The condition $\phi/\psi_{\gamma_1}\in L^2(\bR^d)$ implies that $\phi/\psi_{\gamma_n}, \phi/\psi_0\in L^2(\bR^d)$ as well. There are sufficient conditions, like $\phi$ is bounded and has compact support, leading to it. Moreover, for every $\gamma\geq \gamma_1$,
\[
	\phi(x)=\frac{\psi_{\gamma}^2(x)}{\int \psi_{\gamma}^2(x)dx}
\]
is also an example satisfying \eqref{EQ6RDP}. 
\end{remark}

To prove this theorem, assume without loss of generality that all $X^n$ and $X$ are realized on a common family of probability measure spaces $(\Xi, \mathbf{Q}^x)_{x\in \bR^d}$, where $\Xi$ is a certain measurable space and $\mathbf{Q}^x$ is a probability measure on it (although they are defined on $\Omega$ before Theorem~\ref{THM61}). In other words, for $\varpi\in \Xi$, $t\mapsto X^n_t(\varpi)$ or $t\mapsto X_t(\varpi)$ forms a c\`adl\`ag path in $\bR^d$, and by letting
\[
\begin{aligned}
	&X^n:\Xi\rightarrow \Omega,\quad \varpi \mapsto X^n_\cdot(\varpi),  \\
	&X:\Xi\rightarrow \Omega,\quad \varpi \mapsto X_\cdot(\varpi),
\end{aligned}
\] 
it holds $\mathbf{P}^x_{n}=\mathbf{Q}^x\circ (X^n)^{-1}$ and $\mathbf{P}^x=\mathbf{Q}^x\circ X^{-1}$ for $x\in \bR^d$. 
Then $\bfQ^\phi(\cdot):=\int_{\bR^d}\bfQ^x(\cdot)\phi(x)dx$ defines a new probability measure on $\Xi$ and \eqref{EQ6NOF} is equivalent to
\begin{equation}\label{EQ6NQP}
\lim_{n\rightarrow \infty}\mathbf{Q}^\phi(f(X^n))=\bfQ^\phi(f(X)). 
\end{equation}
In what follows, the proof of \eqref{EQ6NQP} will be divided into two parts. The first one is to prove the Mosco convergence of the associated Dirichlet forms of $X^n$ and the second is to demonstrate the tightness of $X^n$.

\subsection{Mosco convergence}

The conception of Mosco convergence is reviewed in Appendix~\ref{SECB} (see Definition~\ref{DEFB6}). Recall that the Dirichlet form of $X^n$ (resp. $X$) is $(\sE^{(\gamma_n)},\sF^{(\gamma_n)})$ on $L^2(\bR^d,\fm_{\gamma_n})$ (resp. $(\sE^{(0)},\sF^{(0)})$ on $L^2(\bR^d,\fm_{0})$). Denote analogically
\begin{equation}\label{EQ6HNL}
H_n:=L^2(\bR^d,\fm_{\gamma_n}),\quad (\sE^n,\sF^n):=(\sE^{(\gamma_n)},\sF^{(\gamma_n)}),
\end{equation}
and
\begin{equation}\label{EQ6HLR}
H:=L^2(\bR^d,\fm_0),\quad (\sE,\sF):=(\sE^{(0)},\sF^{(0)}). 
\end{equation}
As explained in Remark~\ref{RMB2}, $H_n$ converges to $H$ in the sense of Definition~\ref{DEFB1}. Set $\cH:=\{H_n,H:n\geq 1\}$. Then it is sensible to explore Mosco convergence working on them. Since this result is of independent interest as mentioned in \S\ref{SEC415}, we conclude it as a theorem. 





\begin{theorem}\label{THM63}
Let $(\mathscr{E}^n, \mathscr{F}^n)$ and $(\sE,\sF)$ be the Dirichlet forms in \eqref{EQ6HNL} and \eqref{EQ6HLR}. Then $\mathscr{E}^n$ converges to $\mathscr{E}$ in the sense of Mosco.
\end{theorem}
\begin{proof}
To show the condition (M1), let $f_n$ be a sequence converging to $f$ weakly in $\cH$ in the sense of Definition~\ref{DEFB3} and suppose $\lim_n\mathscr{E}^n(f_n,f_n)<\infty$ without loss of generality. It suffices to prove 
\begin{equation}\label{EQ6EFF}
\sE(f,f)\leq \lim_n\mathscr{E}^n(f_n,f_n).
\end{equation}
To this end, denote 
\[
	J_n(x, y)=\frac{\psi_{\gamma_n}(x)\psi_{\gamma_n}(y)}{|x-y|^{(d+\alpha)}},\quad  J(x, y)=\frac{\psi_{0}(x)\psi_{0}(y)}{|x-y|^{(d+\alpha)}}.
\] 
Put $\bar{f}_n(x, y)=\big(f_n(x)-f_n(y)\big)\sqrt{J_n(x, y)}$ for $(x, y)\in\mathbb{R}^d\times\mathbb{R}^d\setminus D$, which form a bounded sequence in $L^2:=L^2(\mathbb{R}^d\times\mathbb{R}^d\setminus D, dxdy)$, and thus there is a subsequence, still denoted by $\{\bar{f}_n\}$, converging to some function $\bar{f}$ weakly in $L^2$. 
We claim that 
$$\bar{f}(x, y)=\big(f(x)-f(y)\big)\sqrt{J(x, y)}=:\tilde{f}(x,y), \quad dxdy\text{-a.e.},$$
which leads to \eqref{EQ6EFF} since 
\[
\sE(f,f)=\frac{c_{\alpha,d}}{2}\|\bar{f}\|^2_{L^2}\leq \frac{c_{\alpha,d}}{2}\liminf_n \|\bar{f}_n\|^2_{L^2}=\lim_n \sE^n(f_n,f_n). 
\]
Indeed, take an arbitrary non-negative function $g\in C_c(\mathbb{R}^d\times\mathbb{R}^d\setminus D)$. Then there is a constant $R>0$ such that 
\begin{equation}\label{EQ6GXY}
\text{supp}[g]\subset \{(x,y): |x|<R,|y|<R, |x-y|>1/R\}.
\end{equation} 
For any $n$, we have
\[
\begin{aligned}
\bigg|\int&\left(\bar{f}(x, y)-\tilde{f}(x,y)\right)g(x, y)dxdy\bigg| \\
	&\leq \left|\int\left(\bar{f}(x, y)-\bar{f}_n(x,y)\right)g(x, y)dxdy\right|+\left|\int\left(\bar{f}_n(x, y)-\tilde{f}(x,y)\right)g(x, y)dxdy\right|.
\end{aligned}
\]
The first term on the right hand side converges to $0$ as $n\rightarrow \infty$, since $\bar{f}_n$ converges to $\bar{f}$ weakly in $L^2$. Denote the second term by $\mathscr{I}_n$. Note that
\[
\begin{aligned}
\sI_n\leq \bigg|&\int\left(f_n(x)\sqrt{J_n(x, y)}-f(x)\sqrt{J(x, y)}\right)g(x, y)dxdy\bigg| \\
 &+\bigg|\int\left(f_n(y)\sqrt{J_n(x, y)}-f(y)\sqrt{J(x, y)}\right)g(x, y)dxdy\bigg|=:\sI^1_n+\sI^2_n. 
\end{aligned}\]
We only need to show $\sI^1_n\rightarrow 0$ and then $\sI^2_n\rightarrow 0$ is analogical. Clearly, $\sI^1_n$ is not greater than
\[
\begin{aligned}
\bigg|\int & f_n(x)\psi_{\gamma_n}(x)\left(\sqrt{\frac{\psi_{\gamma_n}(y)}{\psi_{\gamma_n}(x)}}-\sqrt{\frac{\psi_{0}(y)}{\psi_{0}(x)}}\right)\frac{g(x, y)}{|x-y|^{\frac{d+\alpha}{2}}}dxdy\bigg| \\
&+ \bigg|\int \left(f_n(x)\psi_{\gamma_n}(x)-f(x)\psi_{0}(x)\right)\sqrt{\frac{\psi_{0}(y)}{\psi_{0}(x)}}\frac{g(x, y)}{|x-y|^{\frac{d+\alpha}{2}}} dxdy\bigg|=:\sI^{1,1}_n+\sI^{1,2}_n. 
\end{aligned}\]
It follows from Cauchy-Schwartz inequality and \eqref{EQ6GXY} that 
\begin{equation}\label{EQ6INF}
\begin{aligned}
\sI^{1,1}_n&\leq \|f_n\|_{H_n} \cdot \left(\int \left(\sqrt{\frac{w_{\gamma_n,0}(y)}{w_{\gamma_n,0}(x)}}-1 \right)^2\frac{\psi_0(y)}{\psi_0(x)}\frac{g(x,y)^2}{|x-y|^{d+\alpha}}dxdy\right)^{1/2} \\
&\leq \|f_n\|_{H_n} \|g\|_\infty R^{\frac{d+\alpha}{2}} \left(\int_{B(R)\times B(R)} \left(\sqrt{\frac{w_{\gamma_n,0}(y)}{w_{\gamma_n,0}(x)}}-1 \right)^2\frac{\psi_0(y)}{\psi_0(x)}dxdy\right)^{1/2},
\end{aligned}\end{equation}
where $w_{\gamma_n,0}=\psi_{\gamma_n}/\psi_0$ is positive and continuous on $\bR^d$ due to Lemma~\ref{LMB}. Note that $\sup_n\|f_n\|_{H_n}<\infty$ by Lemma~\ref{LMB4}~(3) and $\psi_0(y)/\psi_0(x)$ is clearly integrable on $B(R)\times B(R)$. In addition, $w_{\gamma_n,0}(y)/w_{\gamma_n,0}(x)$ is bounded on $B(R)\times B(R)$ and for every $x,y\in B(R)$, one can easily deduce from \eqref{EQAULXA} that
\[
	\lim_{\gamma_n\downarrow 0} \frac{w_{\gamma_n,0}(y)}{w_{\gamma_n,0}(x)}=1. 
\]
Hence by applying the dominated convergence theorem to the last term in \eqref{EQ6INF}, we obtain $\sI^{1,1}_n\rightarrow 0$. On the other hand, mimicking \eqref{EQ6INF}, one can figure out
\[
	x\mapsto \int_{\bR^d}\sqrt{\frac{\psi_{0}(y)}{\psi_{0}(x)}}\frac{g(x, y)}{|x-y|^{\frac{d+\alpha}{2}}}dy \in L^2(\bR^d). 
\]
Since $f_n\psi_{\gamma_n}\rightarrow f\psi_{0}$ weakly in $L^2(\mathbb{R}^{d})$ by Lemma~\ref{LMB4}~(4), we obtain $\sI^{1,2}_n\rightarrow 0$. Eventually for all $g\in C_c(\bR^d\times \bR^d\setminus D)$, it holds
\[
\bigg|\int\left(\bar{f}(x, y)-\tilde{f}(x,y)\right)g(x, y)dxdy\bigg|=0,
\]
which leads to $\bar{f}=\tilde{f}$. Therefore, (M1) is verified. 

Now we turn to verify (M2). When $f\in H\setminus \sF$, take $f_n:=f$. Since $f_n$ clearly converges to $f$ weakly in $\cH$, it follows from (M1) that
\[
	\infty=\sE(f)\leq \liminf_n \sE^n(f_n)=\lim_n\sE^n(f_n)=\infty. 
\] 
Next fix $f\in \sF$. Consider the Hilbert spaces $H':=\sF$ endowed with the norm $\|\cdot\|_{\sE_1}$ and $H'_n:=\sF^n$ endowed with the norm $\|\cdot\|_{\sE^n_1}$. It is straightforward to verify that $H'_n$ converges to $H'$ in the sense of Definition~\ref{DEFB1} by taking $C=C_c^\infty(\bR^d)$ and $\Phi_n=\text{id}$. Applying Lemma~\ref{LMB4}~(1) to $\cH':=\{H'_n,H':n\geq 1\}$, we can take a sequence $f_n\in H'_n$ such that $f_n$ converges to $f$ strongly in $\cH'$. Particularly, $f_n$ also converges to $f$ strongly in $\cH$. Consequently, it follows from Lemma~\ref{LMB4}~(2) that
\[
	\|f_n\|_{H_n}\rightarrow \|f\|_H,\quad \|f_n\|_{H'_n}\rightarrow \|f\|_{H'}. 
\]
As a result, 
\[
	\lim_n \sE^n(f_n,f_n)=\lim_n \sE^n_1(f_n,f_n)-\lim_n \|f_n\|_{H_n}^2=\sE_1(f,f)-\|f\|^2_H=\sE(f,f),
\]
which leads to (M2). That completes the proof.
\end{proof}

Following Theorem~\ref{THMB}, we can conclude the convergence of associated semigroups and resolvents. Recall that $Q^n_t:=Q^{(\gamma_n)}_t$ and $R^n_\lambda:=R^{(\gamma_n)}_\lambda$ are the semigroup and resolvent of $(\sE^n,\sF^n)$ respectively. Both $Q^n_t$ and $R^n_\lambda$ ($t\geq 0,\lambda>0$) are bounded linear operators on $H_n$. Set $Q_t:=Q^{(0)}_t$ and $R_\lambda:=R^{(0)}_\lambda$ further. Note that $R^n_\lambda f\in \sF^n$ and $R_\lambda f\in \sF\subset \sF^n$ for all $f\in H\subset H_n$. 

\begin{corollary}\label{COR63}
For any $t\geq 0$ and $\lambda>0$, $Q^n_t\rightarrow Q_t$ and $R^n_\lambda\rightarrow R_\lambda$ as $n\rightarrow \infty$ in the sense of Definition~\ref{DEFB3}~(3). Furthermore, for any $\lambda>0$ and $f\in H$, 
\begin{equation}\label{EQ6NEN}
\lim_{n\rightarrow\infty}\sE^n(R^n_\lambda f-R_\lambda f, R^n_\lambda f-R_\lambda f)=0.
\end{equation}
\end{corollary}
\begin{proof}
It suffices to prove \eqref{EQ6NEN}. Note that $\sE^n(R^n_\lambda f-R_\lambda f, R^n_\lambda f-R_\lambda f)$ is equal to
\[
	\sE^n(R_\lambda f,R_\lambda f)-\lambda \|R^n_\lambda f\|_{H_n}^2+(f,R^n_\lambda f)_{H_n}-2(R_\lambda f,f)_{H_n}+2\lambda (R_\lambda f,R^n_\lambda f)_{H_n}. 
\]
Since $R^n_\lambda f$ converges to $R_\lambda f$ in $\cH$, it follows from Lemma~\ref{LMB4} that $\|R^n_\lambda f\|_{H_n}\rightarrow \|R_\lambda f\|_{H}$, $(f,R^n_\lambda f)_{H_n}\rightarrow (f,R_\lambda f)_{H}$ and $(R_\lambda f,R^n_\lambda f)_{H_n}\rightarrow (R_\lambda f,R_\lambda f)_{H}$. In addition, the dominated convergence theorem yields $\sE^n(R_\lambda f,R_\lambda f)\rightarrow \sE(R_\lambda f,R_\lambda f)$ and $(R_\lambda f,f)_{H_n}\rightarrow (R_\lambda f,f)_{H}$. Finally we can conclude that
\[
\lim_{n\rightarrow\infty}\sE^n(R^n_\lambda f-R_\lambda f, R^n_\lambda f-R_\lambda f)=\sE(R_\lambda f-R_\lambda f, R_\lambda f-R_\lambda f)=0.
\]
That completes the proof. 
\end{proof}

\subsection{Proof of Theorem~\ref{THM61}}


Denote the $1$-capacity of $(\sE^n,\sF^n)$ (resp. $(\sE,\sF)$) by $\text{Cap}^n$ (resp. $\text{Cap}$). Note that $\text{Cap}^n(A)\leq \text{Cap}(A)$ for any Borel set $A\subset \bR^d$. Let us prepare a simple lemma as below.

\begin{lemma}\label{LM64}
For any nearly Borel measurable set $G$ with $\text{Cap}^n(G)<\infty$, set $\sigma^n_G:=\{t>0:X^n_t\in G\}$. Then it holds
\[
	\bfQ^\phi(\mathrm{e}^{-\sigma^n_G})\leq C_\phi \text{Cap}^n(G)^{1/2},
\]
where $C_\phi:=\|\phi/\psi_{\gamma_1}\|_{L^2(\bR^d)}$ is a finite constant independent of $n$. 
\end{lemma}
\begin{proof}
Set $w_1(\cdot):=\bfQ^\cdot (\mathrm{e}^{-\sigma^n_G})$. It follows from \cite[Theorem~4.2.5]{bib11} that $w_1$ is a quasi-continuous function in $\sF^n$. Clearly $\phi/\psi_{\gamma_n}^2\in H_n$ by \eqref{EQ6RDP} and Remark~\ref{RM62}. Applying \cite[Theorem~2.1.5]{bib11}, we can deduce that
\[
\begin{aligned}
\bfQ^\phi(\mathrm{e}^{-\sigma^n_G})&=\int_{\bR^d}w_1(x)\frac{\phi(x)}{\psi_{\gamma_n}(x)^2}\fm_{\gamma_n}(dx)=\sE^n_1(w_1, R^n_1\left(\phi/\psi_{\gamma_n}^2\right)) \\
&\leq \sE^n_1(R^n_1(\phi/\psi_{\gamma_n}^2),R^n_1(\phi/\psi_{\gamma_n}^2))^{1/2}\cdot \sE^n_1(w_1,w_1)^{1/2} \\
&=\left(\int_{\bR^d} \phi(x)R^n_1(\phi/\psi_{\gamma_n}^2)(x)dx\right)^{1/2} \text{Cap}^n(G)^{1/2}. 
\end{aligned}\]	
It suffices to show $C_{\phi,n}:=\left(\int_{\bR^d} \phi(x)R^n_1(\phi/\psi_{\gamma_n}^2)(x)dx\right)^{1/2}\leq C_\phi$. Indeed, it follows from $\phi/\psi_{\gamma_n}^2\in H_n$ that
\[
\begin{aligned}
C_{\phi,n}^2 &=(\phi/\psi_{\gamma_n}^2, R^n_1(\phi/\psi_{\gamma_n}^2))_{H_n}=\int_0^\infty \mathrm{e}^{-t}\cdot (\phi/\psi_{\gamma_n}^2, Q^n_t(\phi/\psi_{\gamma_n}^2))_{H_n} dt \\ 
&\leq (\phi/\psi_{\gamma_n}^2, \phi/\psi_{\gamma_n}^2)_{H_n}=\int \left(\phi/\psi_{\gamma_n}\right)^2(x)dx\leq \int \left(\phi/\psi_{\gamma_1}\right)^2(x)dx=C_\phi^2.
\end{aligned}\]
The first inequality is led by
\[
(\phi/\psi_{\gamma_n}^2, Q^n_t(\phi/\psi_{\gamma_n}^2))_{H_n}=(Q^n_{t/2}(\phi/\psi_{\gamma_n}^2), Q^n_{t/2}(\phi/\psi_{\gamma_n}^2))_{H_n}\leq (\phi/\psi_{\gamma_n}^2, \phi/\psi_{\gamma_n}^2)_{H_n}.  
\] 
That completes the proof. 
\end{proof}

We pursue the proof of Theorem~\ref{THM61}. The idea of it is due to \cite{bib4} and the crucial fact is that $\psi_{\gamma_n}$ is monotone in $n$. 

\begin{proof}[Proof of Theorem~\ref{THM61}]
\emph{Step 1.} We first show for any $\lambda, T>0$ and any bounded $h\in H$, 
\begin{equation}\label{EQ6NQP2}
\lim_{n\rightarrow \infty} \bfQ^\phi\left[\sup_{t\in [0,T]} \left|R^n_\lambda h(X^n_t)-R_\lambda h(X^n_t)\right|\right]=0.
\end{equation}
Since $R^n_\lambda h, R_\lambda h\in \sF^n$, suppose they are taken to be $\sE^n$-quasi-continuous versions still denoted by $R^n_\lambda h$ and $R_\lambda h$. Moreover, $R^n_\lambda h-R_\lambda h\in \sF^n$ leads to $|R^n_\lambda h-R_\lambda h|\in \sF^n$ and $\sE^n(|R^n_\lambda h-R_\lambda h|,|R^n_\lambda h-R_\lambda h|)\leq \sE^n(R^n_\lambda h-R_\lambda h,R^n_\lambda h-R_\lambda h)$. Fix an arbitrary small constant $\varepsilon>0$. Set $G^n_\varepsilon:=\{x: |R^n_\lambda h(x)-R_\lambda h(x)|>\varepsilon\}$. Then $G^n_\varepsilon$ is $\sE^n$-q.e. finely open with 
\[
	\text{Cap}^n(G^n_\varepsilon)\leq \frac{1}{\varepsilon^2}\sE^n(|R^n_\lambda h-R_\lambda h|,|R^n_\lambda h-R_\lambda h|)\leq \frac{1}{\varepsilon^2}\sE^n(R^n_\lambda h-R_\lambda h,R^n_\lambda h-R_\lambda h).
\]
Further set $\sigma_{G^n_\varepsilon}:=\{t>0: X_t^n\in G_\varepsilon^n\}$. Then
\begin{equation}\label{EQ6QPT}
\bfQ^\phi\left[\sup_{t\in [0,T]} \left|R^n_\lambda h(X^n_t)-R_\lambda h(X^n_t)\right|; T<\sigma_{G^n_\varepsilon}\right]\leq \varepsilon\cdot \bfQ^\phi (T<\sigma_{G^n_\varepsilon})\leq \varepsilon,
\end{equation}
and it follows from $\|R^n_\lambda h\|_\infty \leq \frac{1}{\lambda}\|h\|_\infty$ and   $\|R_\lambda h\|_\infty \leq \frac{1}{\lambda}\|h\|_\infty$ that
\[
\begin{aligned}
\bfQ^\phi\bigg[\sup_{t\in [0,T]}& \left|R^n_\lambda h(X^n_t)-R_\lambda h(X^n_t)\right|; T\geq \sigma_{G^n_\varepsilon}\bigg]\\ 
&\leq \frac{2\|h\|_\infty}{\lambda} \bfQ^\phi(T\geq \sigma_{G^n_\varepsilon}) \leq \frac{2\|h\|_\infty \mathrm{e}^T}{\lambda} \bfQ^\phi(\mathrm{e}^{-\sigma_{G^n_\varepsilon}}). 
\end{aligned}\]
Applying Lemma~\ref{LM64} to $G^n_\varepsilon$, we obtain the above term is not greater than
\begin{equation}\label{EQ6HET}
\frac{2\|h\|_\infty \mathrm{e}^T C_\phi}{\lambda}\text{Cap}^n(G^n_\varepsilon)^{1/2}\leq \frac{2\|h\|_\infty \mathrm{e}^T C_\phi}{\lambda}\frac{\sE^n(R^n_\lambda h-R_\lambda h,R^n_\lambda h-R_\lambda h)^{1/2}}{\varepsilon}.
\end{equation}
By \eqref{EQ6NEN}, there exists $N\in \bN$ such that for all $n>N$, 
\[
\frac{2\|h\|_\infty \mathrm{e}^T C_\phi}{\lambda}\cdot \sE^n(R^n_\lambda h-R_\lambda h,R^n_\lambda h-R_\lambda h)^{1/2}\leq \varepsilon^2. 
\] 
As a result, \eqref{EQ6QPT} and \eqref{EQ6HET} yield 
\[
	\bfQ^\phi\left[\sup_{t\in [0,T]} \left|R^n_\lambda h(X^n_t)-R_\lambda h(X^n_t)\right|\right]\leq 2\varepsilon,\quad \forall n>N. 
\]
This leads to \eqref{EQ6NQP2}. 

\emph{Step 2.} Fix $g\in C_c(\bR^d)$ and $T,\varepsilon>0$. We claim that there exists $h\in C_c(\bR^d)$, a constant $\lambda_0>0$ and an integer $N$ such that 
\begin{equation}\label{EQ6NNQ}
	\sup_{n\geq N}\mathbf{Q}^\phi\bigg[ \sup_{t\in[0,T]}\left|\lambda_0R_{\lambda_0}^nh(X_t^n)-g(X_t^n)\right|\bigg] <\varepsilon.
\end{equation}
To this end, take $h\in \sF\cap C_c(\bR^d)$ such that $\|h\|_\infty\leq 2\|g\|_\infty$ and 
\begin{equation}\label{EQ6XRD}
	\sup_{x\in\bR^d}|h(x)-g(x)|<\varepsilon/3. 
\end{equation}
The existence of $h$ is due to the regularity of $(\sE,\sF)$. Note that $\lambda R_\lambda h\in \sF$ is $\sE$-quasi-continuous and $\lambda R_\lambda h$ converges to $h$ strongly in $\sF$ endowed with the $\sE_1$-norm by \cite[Lemma~1.3.3]{bib11}. Thus applying \cite[Theorem~1.3.3 and Exercise~1.3.16]{bib19}, we can take a $\text{Cap}$-nest $\{F_m:m\geq 1\}$, i.e. $F_m$ is a sequence of increasing closed sets such that $\lim_{m\rightarrow \infty}\text{Cap}(F^c_m)=0$, such that $\lambda R_\lambda h$ converges to $h$ uniformly on each $F_m$. Take $m_0\in \bN$ such that 
\begin{equation}\label{EQ6CFC}
	\text{Cap}(F^c_{m_0})\leq \frac{\varepsilon^2}{\left(24\|g\|_\infty \mathrm{e}^T C_\phi\right)^2},
\end{equation}
where $C_\phi$ is the constant in Lemma~\ref{LM64} and further take $\lambda_0$ such that 
\begin{equation}\label{EQ6XFM}
	\sup_{x\in F_{m_0}}|\lambda_0 R_{\lambda_0}h(x)-h(x)|<\frac{\varepsilon}{6}. 
\end{equation}
Finally by virtue of \eqref{EQ6NQP2} for $\lambda_0, T$ and $h$, take $N\in \bN$ such that
\begin{equation}\label{EQ6NNQ2}
	\sup_{n\geq N}\bfQ^\phi\left[\sup_{t\in [0,T]} \left|R^n_{\lambda_0} h(X^n_t)-R_{\lambda_0} h(X^n_t)\right|\right]< \frac{\varepsilon}{3\lambda_0}. 
\end{equation}
With $h$, $\lambda_0$ and $N$ in hand, we pursue to verify \eqref{EQ6NNQ}. Clearly, 
\begin{equation}\label{EQ6LRN}
	|\lambda_0 R^n_{\lambda_0} h-g|\leq \lambda_0 |R^n_{\lambda_0}h-R_{\lambda_0}h| + |\lambda_0 R_{\lambda_0}h-h|+|h-g|. 
\end{equation}
From \eqref{EQ6XRD}, we obtain 
\begin{equation}\label{EQ6QPTT}
\bfQ^\phi\left[\sup_{t\in [0,T]} \left|h(X^n_t)-g(X^n_t)\right|\right]<\frac{\varepsilon}{3}
\end{equation}
for all $n$. For every $n$, set $\sigma_{m_0}^n:=\inf\{t>0:X_t^n\in F^c_{m_0}\}$. Then it follows from \eqref{EQ6XFM} that 
\[
	\bfQ^\phi\left[\sup_{t\in [0,T]} \left|\lambda_0R_{\lambda_0}h(X^n_t)-h(X^n_t)\right|; T<\sigma^n_{m_0}\right] < \frac{\varepsilon}{6}. 
\]
In addition, mimicking the first step, one can deduce from $\|h\|_\infty\leq 2\|g\|_\infty$, $\text{Cap}^n\leq \text{Cap}$ and \eqref{EQ6CFC} that
\[
\begin{aligned}
\bfQ^\phi\bigg[\sup_{t\in [0,T]}& \left|\lambda_0R_{\lambda_0}h(X^n_t)-h(X^n_t)\right|; T\geq \sigma^n_{m_0}\bigg] \\
 &\leq \left(\|\lambda_0R_{\lambda_0}h\|_\infty +\|h\|_\infty\right)\bfQ^\phi(T\geq \sigma^n_{m_0}) \\
 &\leq 2\|h\|_\infty \mathrm{e}^T C_\phi \text{Cap}^n(F^c_{m_0})^{1/2} \\
 &\leq 4\|g\|_\infty \mathrm{e}^T C_\phi \text{Cap}(F^c_{m_0})^{1/2}<\frac{\varepsilon}{6}. 
\end{aligned}\]
Consequently, for all $n$,
\begin{equation}\label{EQ6QHT}
\bfQ^\phi\left[\sup_{t\in [0,T]} \left|\lambda_0R_{\lambda_0}h(X^n_t)-h(X^n_t)\right|\right]<\frac{\varepsilon}{3}. 
\end{equation}
Eventually, \eqref{EQ6NNQ2}, \eqref{EQ6LRN}, \eqref{EQ6QPTT} and \eqref{EQ6QHT} yield  \eqref{EQ6NNQ}. 

\emph{Step 3.} In this step, we will demonstrate that for any $m\geq 1$ and $g_i\in C_c(\bR^d)$ for $1\leq i\leq m$,
\begin{equation}\label{EQ6XNG}
	\{\xi_n:=\left(g_1(X^n),\cdots, g_m(X^n)\right): n\geq 1\}
\end{equation}
under $\bfQ^\phi$ forms a tight family on $D([0,\infty),\bR^m)$, i.e. the Skorohod topology space on $\bR^m$, by virtue of \cite[Theorem~3.9.4 and Remark~3.9.5~(b)]{bib12}. More precisely, 
\[
	\xi_n: \Xi\rightarrow D([0,\infty),\bR^m),\quad \varpi\mapsto \xi_n\varpi,
\]
where $\xi_n\varpi(t):=\left(g_1(X^n_t(\varpi)),\cdots, g_m(X^n_t(\varpi))\right)\in \bR^m$ is clearly c\`adl\`ag in $t$, induces a probability measure $\bfQ^\phi\circ \xi_n^{-1}$ on $D([0,\infty),\bR^m)$. Our object is to show $\{\bfQ^\phi\circ \xi_n^{-1}: n\geq 1\}$ is tight. It suffices to consider $m=1$ and write $g:=g_1$ for the sake of brevity. To this end, fix $\varepsilon, T>0$. Apply Step 2 to these $g, \varepsilon, T$ and take $h, \lambda_0, N$ such that \eqref{EQ6NNQ} holds. Set
\[
	Y^n_t:=\lambda_0 R^n_{\lambda_0}h(X^n_t),\quad Z^n_t:=\lambda_0\left(\lambda_0R^n_{\lambda_0} h-h \right)(X_t^n). 
\]
From the Fukushima's decomposition of $X^n$ with respect to $\lambda_0 R^n_{\lambda_0}h$ (see \cite[Theorem~5.2.5]{bib11}), one can find that 
\[
	t\mapsto Y^n_t-\int_0^t Z^n_sds
\]
is a martingale relative to the filtration of $X^n$. In addition, \eqref{EQ6NNQ} tells us
\[
\sup_{n\geq N} \bfQ^\phi \left[\sup_{t\in [0,T]}\left|Y^n_t-g(X^n_t) \right| \right]<\varepsilon. 
\]
Furthermore, since $\|h\|_\infty\leq 2\|g\|_\infty<\infty$, it follows that
\[
\sup_{n\geq 1}\bfQ^\phi\left[\sup_{t\in [0,T]}|Z^n_t|\right]\leq 2\lambda_0 \|h\|_\infty<\infty. 
\]
Eventually, \cite[Theorem~3.9.4 and Remark~3.9.5~(b)]{bib12} yield the desirable tightness. 

\emph{Step 4.} Finally, we shall conclude that $\xi_n$ in \eqref{EQ6XNG} is weakly convergent to $\xi:=(g_1(X),\cdots,g_m(X))$ under $\bfQ^\phi$. Since $C_c(\bR^d)$ strongly separates points in $\bR^d$ (for the definition of strong separation, see \cite[\S3.4]{bib12}; for the proof of this fact, see \cite{bib25}), this convergence leads to  \eqref{EQ6NQP} due to \cite[Corollary~3.9.2]{bib12}. To show $\xi_n$ weakly converges to $\xi$, note that $\{\xi_n:n\geq 1\}$ is tight by the third step. Thus it suffices to show the finite dimensional distributions of $\xi_n$ are weakly convergent to those of $\xi$ by employing \cite[Theorem~3.7.8]{bib12}. To this end, take $f_1\in C_b(\bR^m)$ and set $h_1:=f_1\circ (g_1,\cdots, g_m)\in C_b(\bR^d)$. Note that $C:=f_1(0)$ is not necessarily equal to $0$. For every $x\notin \cup_{1\leq i\leq m}\text{supp}[g_i]$, we have $h_1(x)=C$. Then $\tilde{h}_1:=h_1-C$ defines a continuous function with compact support and thus $\tilde{h}_1\in H$. For any $t_1>0$, 
\[
	\bfQ^\phi(h_1(X^n_{t_1}))=\bfQ^\phi(\tilde{h}_1(X^n_{t_1}))+C=\left(Q^n_{t_1}\tilde{h}_1, \frac{\phi}{\psi_{\gamma_n}^2}\right)_{H_n}+C. 
\]
Corollary~\ref{COR63} tells us $Q^n_{t_1} \tilde{h}_1$ converges to $Q_{t_1}\tilde{h}_1$ strongly (as well as weakly due to Remark~\ref{RMB5}) in $\cH$. Clearly, $\phi/\psi_{\gamma_n}^2$ converges to $\phi/\psi_0^2$ strongly in $\cH$ by means of Lemma~\ref{LMB4}~(4) and $\phi/\psi_{\gamma_1}\in L^2(\bR^d)$. As a result, 
\[
	\bfQ^\phi(h_1(X^n_{t_1}))\rightarrow \left(Q_{t_1}\tilde{h}_1, \frac{\phi}{\psi_{0}^2}\right)_{H}+C=\bfQ^\phi(h_1(X_{t_1})). 
\]
Mimicking the above argument, one can obtain by induction that for $t_1<\cdots< t_k$, $f_1,\cdots, f_k\in C_b(\bR^m)$, it still holds (see \cite[Theorem~3.7]{bib4} for more details)
\[
	\lim_{n\rightarrow 0}\mathbf{Q}^\phi\big[ h_1(X_{t_1}^n)\cdots h_k(X_{t_k}^n)\big ]=\mathbf{Q}^\phi\big[ h_1(X_{t_1})\cdots h_k(X_{t_k})\big ],
\]
where $h_i:=f_i\circ (g_1,\cdots, g_m)$. In other words, the finite dimensional distribution of $\xi_n$ at $(t_1,\cdots, t_k)$ is weakly convergent to that of $\xi$. That completes the proof. 
\end{proof}

\appendix

\section{Basics of isotropic $\alpha$-stable process}

Fix $0<\alpha<2\wedge d$. The generator of isotropic $\alpha$-stale process is denoted by $\Delta^{\alpha/2}$ (on $L^2(\bR^d)$), whose definition is usually given by Fourier transform. 
We present two alternative equivalent definitions of $\Delta^{\alpha/2}$ as follows. Note that $p_t(x):=p(t,0,x)$ stands for its transition density, which enjoys the following properties:
\begin{equation}\label{EQAPTX}
p_t(x)=t^{-d/\alpha}p_1(x/t^{1/\alpha}),\quad t>0, x\neq 0,
\end{equation}
and 
\begin{equation}\label{EQAPCR}
	p_1\in C^\infty(\bR^d)\cap \mathcal{B}_+(\bR^d),\quad p_1(x) \approx 1\wedge |x|^{-d-\alpha};
\end{equation}
see e.g. \cite[\S2.6]{bib20}.

\begin{definition}\label{DEFA1}
Recall that the constant $c_{\alpha,d}$ is given in \S\ref{SEC2}. Define
\[
	L_If:=\lim_{r\downarrow 0} c_{\alpha, d} \int_{y: |y|>r} \frac{f(\cdot+y)-f(\cdot)}{|y|^{d+\alpha}}dy
\]
with the limit in $L^2(\bR^d)$, where the domain $\cD(L_I)$ consists of functions such that this limit exists. In addition, define
\[
	L_Sf:= \lim_{t\downarrow 0}\frac{1}{t}\left(\int_{\bR^d}p_t(\cdot-y)f(y)dy-f(\cdot) \right)
\]
with the limit in $L^2(\mathbb{R}^d)$, where the domain $\cD(L_S)$ consists of functions such that this limit exists. 
\end{definition}
\begin{remark}
Note that $L_I=L_S=\Delta^{\alpha/2}$ with $\cD(L_I)=\cD(L_S)=\cD(\Delta^{\alpha/2})=H^\alpha(\bR^d)$, where $H^\alpha(\bR^d)$ is the Sobolev space of order $\alpha$.
\end{remark}

Let $u_\lambda$ given by \eqref{EQ2ULX} be the resolvent density of isotropic $\alpha$-stable process. The following lemma summarizes some crucial properties of $u_\lambda$. Though it is elementary, we present a proof for readers' convenience. 

\begin{lemma}\label{LMB}
 Fix $\lambda\geq 0$ and let $u_\lambda$ be in \eqref{EQ2ULX} or \eqref{EQ2UXP}. Then the following hold:
\begin{itemize}
\item[(1)] $u_\lambda\in C^\infty(\bR^d\setminus \{0\})$ and $u_\lambda(x)>0$ for all $x\in \bR^d\setminus \{0\}$.
\item[(2)] For any $\mu\geq 0$, 
\begin{equation}\label{EQAXUL}
	\lim_{x\rightarrow 0}\frac{u_{\lambda+\mu}(x)}{u_\lambda(x)}=1. 
\end{equation}
Particularly, $w_{\lambda, \mu}(x):=u_{\lambda+\mu}(x)/u_\mu(x)$ for $x\neq 0$ and $w_{\lambda, \mu}(0):=1$ form a continuous function on $\bR^d$. 
\item[(3)] The following limit holds uniformly on any compact subset of $\bR^d\setminus \{0\}$: 
\[
	\lim_{t\downarrow 0} \frac{1}{t}\left(p_t\ast u_\lambda -u_\lambda\right)=\lambda u_\lambda. 
\]
\end{itemize}
\end{lemma}
\begin{proof}
For the first assertion, it suffices to consider $\lambda>0$. We first prove $u_\lambda(x)>0$ for $x\neq 0$. Argue by contradiction and suppose that $u_\lambda(x_0)=0$ for some $x_0\neq 0$. It follows from \eqref{EQAPTX} that
\begin{equation}\label{EQBULX}
u_\lambda(x)=\int_0^\infty \mathrm{e}^{-\lambda t} t^{-\frac{d}{\alpha}} p_1\left(\frac{x}{t^{1/\alpha}}\right)dt,\quad x\neq 0.
\end{equation} 
Then from \eqref{EQBULX} and the smoothness of $p_1$, we obtain $p_1(x_0/t^{1/\alpha})=0$ for all $t>0$. 
Note that $p_1$ is a radius function by the isotropy of $\alpha$-stable process, i.e. there is a function $r_1$ on $[0,\infty)$ such that $p_1(x)=r_1(|x|)$. This implies $p_1(y)=0$ for all $y\neq 0$, which leads to a contradiction with $\int_{\bR^d}p_1(y)dy=1$. Now we turn to prove $u_\lambda\in C^\infty(\bR^d\setminus \{0\})$. Substituting $\mathfrak{t}:=|x|/t^{1/\alpha}$ in \eqref{EQBULX}, we obtain
\begin{equation}\label{EQAULXA}
\begin{aligned}
	u_\lambda(x)&=\frac{\alpha}{|x|^{d-\alpha}}\int_0^\infty \exp\left\{-\lambda\frac{|x|^\alpha}{\ft^\alpha}\right\}\ft^{d-\alpha-1}p_1\left(\frac{x}{|x|}\ft\right)d\ft.
\end{aligned}\end{equation} 
Since \eqref{EQAPCR} and $p_1\left(\frac{x}{|x|}\ft\right)=r_1(\ft)$ is independent of $x$, one can easily conclude that $|x|^{d-\alpha}u_\lambda(x)$ is smooth at $x\neq 0$. Hence $u_\lambda \in C^\infty(\bR^d\setminus \{0\})$. 

To prove \eqref{EQAXUL}, note that
\[
	\exp\left\{-\lambda\frac{|x|^\alpha}{\ft^\alpha}\right\}\ft^{d-\alpha-1}p_1\left(\frac{x}{|x|}\ft\right)\leq \ft^{d-\alpha-1}r_1\left(\ft\right)
\]
is integrable in $\ft$ on $(0,\infty)$ due to \eqref{EQAPCR} and $\alpha<d$. Then \eqref{EQAULXA} and the dominated convergence theorem yield
\[
\begin{aligned}
	\lim_{x\rightarrow 0} \frac{u_{\lambda+\mu}(x)}{u_\lambda(x)}&= \frac{\lim_{x\rightarrow 0}\int_0^\infty \exp\left\{-(\lambda+\mu)\frac{|x|^\alpha}{\ft^\alpha}\right\}\ft^{d-\alpha-1}r_1\left(\ft\right)d\ft}{\lim_{x\rightarrow 0}\int_0^\infty \exp\left\{-\lambda\frac{|x|^\alpha}{\ft^\alpha}\right\}\ft^{d-\alpha-1}r_1\left(\ft\right)d\ft} =1. 
\end{aligned}\]

Finally, a straightforward computation yields
\[
\frac{1}{t}\left(p_t\ast u_\lambda -u_\lambda\right)=\frac{\mathrm{e}^{\lambda t}-1}{t}\int_t^\infty \mathrm{e}^{-\lambda s} s^{-\frac{d}{\alpha}} p_1\left(\frac{x}{s^{1/\alpha}}\right)ds-\frac{1}{t}\int_0^t \mathrm{e}^{-\lambda s} s^{-\frac{d}{\alpha}} p_1\left(\frac{x}{s^{1/\alpha}}\right)ds. 
\]
Then we can obtain the last assertion by virtue of \eqref{EQAPCR} and the dominated convergence theorem. 
\end{proof}

\section{Mosco convergence with changing speed measures}\label{SECB}

To make the paper more self-contained, we summarize in this appendix some basic conceptions and results concerning Mosco convergence from \cite{bib3}. It is working on a sequence of Hilbert spaces $\{H_n:n\geq 1\}$ converging to another one $H$ in the following sense.

\begin{definition}\label{DEFB1}
A sequence of Hilbert spaces $ \{H_n\}$ is called to converge to another Hilbert space $H$, if there exists a dense subspace $C\subset H$ and a sequence of operators 
$$\Phi_n :C\longrightarrow H_n$$
with the following property:
\begin{equation}\label{EQBNPN}
\lim_{n\rightarrow\infty}\|\Phi_n f\|_{H_n}=\|f\|_H
\end{equation}
for every $f\in C$.
\end{definition}
\begin{remark}\label{RMB2}
In \S\ref{SEC6}, we always take $H_n=L^2(\bR^d,\fm_{\gamma_n}), H=L^2(\bR^d,\fm_0)$ with $\gamma_n\downarrow 0$ and $C=H$ (or $C_c^\infty(\bR^d)$), $\Phi_n:=\text{id}$, i.e. $\Phi_n f:=f$ for all $f\in H$. Note that $H\subset H_n$ since $\psi_{\gamma_n}\leq \psi_0$, and \eqref{EQBNPN} holds due to the monotone convergence theorem. 
\end{remark}

Set $\cH:=\{H_n,H:n\geq 1\}$. The following definition presents elementary convergences in the context of these Hilbert spaces. 

\begin{definition}\label{DEFB3}
\begin{itemize}
\item[(1)] (Strong convergence in $\cH$) We say that $f_n$ converges to $f$ strongly in $\cH$ (as $n\rightarrow \infty$), if  $f_n\in H_n$, $f\in H$ and there exists a sequence $\{\tilde{f}_m\}\subset C$ with the following properties:
\[
	\|\tilde{f}_m-f\|_H\rightarrow 0,\quad \lim_m\limsup_n \|\Phi_n\tilde{f}_m-f_n\|_{H_n}=0.
\]
\item[(2)] (Weak convergence in $\cH$) We say that $f_n$ converges to $f$ weakly in $\cH$ (as $n\rightarrow \infty$), if $f_n\in H_n$, $f\in H$ and
$$(f_n, g_n)_{H_n}\rightarrow (f, g)_H$$
for every sequence $\{g_n\}$ converging to $g$ strongly in $\cH$.
\item[(3)] (Convergence of operators) Given a sequence of bounded linear operators $B_n$ on $H_n$, we say $B_n$ strongly converges to a bounded linear operator $B$ on $H$ (as $n\rightarrow \infty$), if for every sequence $f_n$ converging to $f$ strongly in $\cH$, $B_nf_n$ converges to $Bf$ strongly in $\cH$.  
\end{itemize}
\end{definition}

The lemma below states some important results concerning these convergences.

\begin{lemma}[See \cite{bib3}]\label{LMB4}
Let $H_n$ and $H$ be in Definition~\ref{DEFB1}. Then the following hold:
\begin{itemize}
\item[(1)] For every $f\in H$, there exists a sequence $\{f_n: n\geq 1\}$ such that $f_n\rightarrow f$ strongly in $\cH$.
\item[(2)] If $f_n\rightarrow f$ strongly in $\cH$, then $\|f_n\|_{H_n}\rightarrow \|f\|_{H}$.
\item[(3)] If $f_n\rightarrow f$ weakly in $\cH$, then 
$$\sup_n\|f_n\|_{H_n}< \infty, \quad \|f\|_H\leq \liminf_n \|f_n\|_{H_n}.$$
\item[(4)] Let $H_n:=L^2(\bR^d,\varphi_n(x)^2dx)$ and $H=L^2(\bR^d,\varphi(x)^2dx)$ such that $\varphi_n, \varphi$ are positive in $L^2_{\mathrm{loc}}(\bR^d)$ and $H_n$ converges to $H$ in the sense of Definition~\ref{DEFB1}. Assume that $\int_K \left(\varphi_n(x)-\varphi(x)\right)^2 dx\rightarrow 0$ for any compact set $K\subset \bR^d$. Then $f_n$ converges to $f$ strongly (resp. weakly) in $\cH$, if and only if $f_n\varphi_n$ converges to $f\varphi$ strongly  (resp. weakly) in $L^2(\bR^d)$. 
\end{itemize}
\end{lemma}
\begin{remark}\label{RMB5}
When $\varphi_n:=\psi_{\gamma_n}$ and $\varphi:=\psi_0$ as in Remark~\ref{RMB2}, it is straightforward to verify that all conditions in the fourth assertion are satisfied. Meanwhile, this claim indicates that strong convergence in $\cH$ leads to weak convergence in $\cH$. 
\end{remark}

Now we turn to consider the so-called \emph{Mosco convergence} of closed forms. Identify a quadratic form $(\mathscr{E},\cD(\sE))$ on $H$ (or $H_n$) with the function 
\[
\mathscr{E}(\cdot): H\rightarrow \bar{\bR}:=\bR\cup\{\infty\},\quad  f\mapsto 
\begin{cases}
\mathscr{E} (f,f), & f\in\mathcal{D}(\mathscr{E})\\ \infty  & f\notin\mathcal{D}(\mathscr{E}).
\end{cases}
\]
The following conception is our main concern. 

\begin{definition}\label{DEFB6}
Let $(\sE^n,\cD(\sE^n))$ be a closed form on $H_n$ and $(\sE,\cD(\sE))$ be a closed form on $H$. We say $\sE^n$ converges to $\sE$ in the sense of Mosco, if the following conditions hold:
\begin{itemize}
\item[(M1)] If $f_n$ converges to $f$ weakly in $\cH$, then
$$\mathscr{E}(f)\leq\liminf_n\mathscr{E}^n(f_n).$$
\item[(M2)]For every $f\in H$, there exists a sequence $\{f_n:n\geq 1\}$ converging to $f$ strongly in $\cH$ such that 
$$\mathscr{E}(f)=\lim_n\mathscr{E}^n(f_n).$$
\end{itemize}
\end{definition}

The significance of Mosco convergence is indicated in the following well-known result.

\begin{theorem}[See \cite{bib24} and also \cite{bib3}]\label{THMB}
Let $\{\mathscr{E}^n: H_n\rightarrow\bar{\mathbb{R}}\}$ be a sequence of closed forms and $\mathscr{E}$ be a closed form on $H$. Let $(T^n_t)_{t\geq 0}$ and $(G^n_\lambda)_{\lambda>0}$ (resp. $(T_t)_{t\geq 0}$ and $(G_\lambda)_{\lambda>0}$) be the semigroup and resolvent of $\sE^n$ (resp. $\sE$) respectively. Then the following are all equivalent:
\begin{itemize}
\item[(1)] $\mathscr{E}^n$ converges to $\mathscr{E}$ in the sense of Mosco;
\item[(2)] $G^n_{\lambda}$ strongly converges to $G_\lambda$ for every $\lambda>0$;
\item[(3)] $T^n_{t}$ strongly converges to $T_t$ for every $t>0$.
\end{itemize}
\end{theorem}

\bibliographystyle{abbrv}
\bibliography{DFPM2}

\begin{thebibliography}{10}

\bibitem{bib6}
S.~Albeverio and P.~Kurasov.
\newblock {\em {Singular Perturbations of Differential Operators, Solvable
  Schr{\"o}dinger-type Operators}}.
\newblock Cambridge University Press, 2000.

\bibitem{bib25}
D.~Blount and M.~A. Kouritzin.
\newblock {On convergence determining and separating classes of functions}.
\newblock {\em Stochastic Process. Appl.}, 120(10):1898--1907, 2010.

\bibitem{bib23}
K.~Bogdan, K.~Burdzy, and Z.-Q. Chen.
\newblock {Censored stable processes}.
\newblock {\em Probab. Theory Related Fields}, 127(1):89--152, 2003.

\bibitem{bib19}
Z.-Q. Chen and M.~Fukushima.
\newblock {\em {Symmetric Markov processes, time change, and boundary theory}},
  volume~35 of {\em London Mathematical Society Monographs Series}.
\newblock Princeton University Press, Princeton, NJ, 2012.

\bibitem{bib22}
Z.-Q. Chen and M.~Fukushima.
\newblock {One-point reflection}.
\newblock {\em Stochastic Process. Appl.}, 125(4):1368--1393, 2015.

\bibitem{bib21}
K.~L. Chung and J.~B. Walsh.
\newblock {\em {Markov processes, Brownian motion, and time symmetry}}, volume
  249 of {\em Grundlehren der Mathematischen Wissenschaften [Fundamental
  Principles of Mathematical Sciences]}.
\newblock Springer, New York, New York, NY, second edition, 2005.

\bibitem{bib8}
M.~Cranston, L.~Koralov, S.~Molchanov, and B.~Vainberg.
\newblock {Continuous model for homopolymers}.
\newblock {\em J. Funct. Anal.}, 256(8):2656--2696, 2009.

\bibitem{bib9}
M.~Cranston, L.~Koralov, S.~Molchanov, and B.~Vainberg.
\newblock {A solvable model for homopolymers and self-similarity near the
  critical point}.
\newblock {\em Random Oper. Stochastic Equations}, 18(1):73--95, 2010.

\bibitem{bib10}
M.~Cranston and S.~Molchanov.
\newblock {On phase transitions and limit theorems for homopolymers}.
\newblock In {\em Probability and mathematical physics}, pages 97--112. Amer.
  Math. Soc., Providence, RI, 2007.

\bibitem{bib5}
M.~Cranston, S.~Molchanov, and N.~Squartini.
\newblock {Point potential for the generator of a stable process}.
\newblock {\em J. Funct. Anal.}, 266(3):1238--1256, 2014.

\bibitem{bib18}
S.~Dipierro and E.~Valdinoci.
\newblock {A density property for fractional weighted Sobolev spaces}.
\newblock {\em Atti Accad. Naz. Lincei Cl. Sci. Fis. Mat. Natur.},
  26(4):397--422, 2015.

\bibitem{bib12}
S.~N. Ethier and T.~G. Kurtz.
\newblock {\em {Markov Processes: Characterization and convergence}}.
\newblock John Wiley {\&} Sons, New York, 2009.

\bibitem{bib1}
P.~J. Fitzsimmons and L.~Li.
\newblock {On the Dirichlet form of three-dimensional Brownian motion
  conditioned to hit the origin}.
\newblock {\em Sci. China Math.}, to appear, doi:
  https://doi.org/10.1007/s11425-017-9400-8.

\bibitem{bib11}
M.~Fukushima, Y.~Oshima, and M.~Takeda.
\newblock {\em {Dirichlet forms and symmetric Markov processes}}, volume~19 of
  {\em de Gruyter Studies in Mathematics}.
\newblock Walter de Gruyter {\&} Co., Berlin, extended edition, 2011.

\bibitem{bib15}
G.~Giacomin.
\newblock {\em {Random Polymer Models}}.
\newblock Imperial College Press, 2011.

\bibitem{bib4}
P.~Kim.
\newblock {Weak convergence of censored and reflected stable processes}.
\newblock {\em Stochastic Process. Appl.}, 116(12):1792--1814, 2006.

\bibitem{bib3}
A.~V. Kolesnikov.
\newblock {Convergence of Dirichlet forms with changing speed measures on
  $\mathbb{R}^{d}$}.
\newblock {\em Forum Math.}, 17(2):225--259, 2005.

\bibitem{bib24}
K.~Kuwae and T.~Shioya.
\newblock {Convergence of spectral structures: a functional analytic theory and
  its applications to spectral geometry}.
\newblock {\em Comm. Anal. Geom.}, 11(4):599--673, 2003.

\bibitem{bib20}
M.~Kwa{\'{s}}nicki.
\newblock {Ten equivalent definitions of the fractional laplace operator}.
\newblock {\em Fract. Calc. Appl. Anal.}, 20(1):540.

\bibitem{bib16}
I.~M. Lifshitz, A.~Y. Grosberg, and A.~R. Khokhlov.
\newblock {Some problems of the statistical physics of polymer chains with
  volume interaction}.
\newblock {\em Rev. Mod. Phys.}, 50(3):683--713, 1978.

\bibitem{bib7}
K.-I. Sato.
\newblock {\em {L\'evy processes and infinitely divisible distributions}},
  volume~68 of {\em Cambridge Studies in Advanced Mathematics}.
\newblock Cambridge University Press, Cambridge, 2013.

\end{thebibliography}

\end{document}